\documentclass[preprint,12pt]{elsarticle}
\usepackage[margin=3cm]{geometry}
\usepackage[english]{babel}

\usepackage{amssymb}
\usepackage{doi}
\usepackage{url}
\usepackage{amsmath}
\usepackage{amsthm}
\usepackage{tikz}
\usepackage{float}

\usepackage{quiver}
\usetikzlibrary{angles,calc,quotes,babel,arrows.meta,decorations.markings,plotmarks}

\usepackage{euscript}

\newtheorem{thm}[subsection]{Theorem}
\newtheorem{cor}[subsection]{Corollary}
\newtheorem{lem}[subsection]{Lemma}
\newtheorem{prop}[subsection]{Proposition}

\newtheorem{fact}[subsection]{Fact}
\newtheorem{ex}[subsection]{Example}

\newtheorem{rem}[subsection]{Remark}

\theoremstyle{definition}
\newtheorem{defn}[subsection]{Definition}

\newcommand{\nc}{\newcommand}
\newcommand{\rnc}{\renewcommand}

\rnc{\lim}{\mathsf{lim}}
\nc{\inv}{\mathsf{inv}}
\newcommand\mapsfrom{\mathrel{\reflectbox{\ensuremath{\mapsto}}}}
\nc{\mono}{\hookrightarrow}
\nc{\epi}{\twoheadrightarrow}

\nc{\ext}{\bigwedge\nolimits}

\nc{\BaseRing}{\mathsf{k}}

\nc{\Grp}{\mathsf{Grp}}
\nc{\Ab}{\mathsf{Ab}}
\nc{\Set}{\mathsf{Set}}
\nc{\Pres}{\mathsf{Pres}}
\nc{\Ring}{\mathsf{Ring}}
\nc{\Alg}{\mathsf{Alg}}
\nc{\kAlg}{\Alg(\BaseRing)}
\nc{\Mod}{\mathsf{Mod}}
\nc{\kMod}{\Mod(\BaseRing)}
\nc{\C}{\EuScript{C}}
\nc{\D}{\EuScript{D}}
\nc{\A}{\EuScript{A}}
\nc{\B}{\EuScript{B}}
\nc{\FMod}{\mathsf{FunMod}}
\rnc{\O}{\EuScript{O}}
\nc{\BB}{\text{\sf\reflectbox{B}}}
\nc{\F}{\mathcal{F}}
\nc{\U}{U}

\nc{\Zy}{\EuScript{Z}}
\nc{\Bd}{\EuScript{B}}

\nc{\g}{\mathbf{g}}
\nc{\ab}{\mathrm{ab}}
\nc{\ZF}{{\mathbb{Z}[F]}}
\nc{\ZG}{{\mathbb{Z}[G]}}
\nc{\Z}{\mathbb{Z}}
\nc{\Q}{\mathbb{Q}}
\nc{\f}{\mathbf{f}}
\nc{\h}{\mathbf{h}}
\rnc{\r}{\mathbf{r}}
\nc{\tensorsquare}{\widetilde{{\otimes}}^2}
\nc{{\y}}{{\bf y}}
\nc{{\x}}{{\bf x}}
\rnc{\ker}{\mathsf{ker}}
\nc{\coker}{\mathsf{coker}}
\nc{\Fun}{\mathsf{Fun}}
\nc{\eq}{\mathsf{eq}}
\nc{\im}{\mathsf{im}}
\nc{\proj}{\mathrm{proj}}
\nc{\inj}{\mathrm{inj}}

\nc{\RR}{\mathbb{R}}

\nc{\const}{\mathsf{const}}
\nc{\dom}{\mathsf{dom}}
\nc{\cod}{\mathsf{cod}}
\nc{\Mor}{\mathsf{Mor}}

\rnc{\a}{\mathbf{a}}
\rnc{\b}{\mathbf{b}}
\nc{\s}{\mathbf{s}}
\nc{\Tor}{\mathsf{Tor}}

\nc{\Lie}{\mathcal{L}}
\nc{\Ex}{\Lambda}
\nc{\id}{\mathrm{id}}

\makeatletter
\def\ps@pprintTitle{%
  \let\@oddhead\@empty
  \let\@evenhead\@empty
  \def\@oddfoot{\reset@font\hfil\thepage\hfil}
  \let\@evenfoot\@oddfoot
}
\makeatother

\makeatletter
\newcommand{\smallbullet}{}%
\DeclareRobustCommand\smallbullet{%
  {\mathord{\mathpalette\smallbullet@{0.5}}}%
}
\newcommand{\smallbullet@}[2]{%
  {\vcenter{\hbox{\scalebox{#2}{$\m@th#1\bullet$}}}}%
}
\makeatother

\journal{-}

\begin{document}

\begin{frontmatter}

\title{Functorial languages in homological algebra and lower central series}

\author{
Nikita Golub, Vasily Ionin, Lev Mukoseev
}

\begin{abstract} There is a general phenomenon in algebra that numerous functors of homological significance admit characterization as derived limits of elementary functors defined over categories of free extensions.
We demonstrate that upon restriction to appropriate subcategories of the category of groups, one may express analogously more interesting functors, including homology groups with cyclic coefficients.

Moreover, we are laying the foundations of the so-called \(\f\r_\infty\)-language, extending the \(\f\r\)-language of Roman Mikhailov and Sergei O. Ivanov. This language is constructed by augmenting the \(\f\r\)-language through the introduction of an infinite family of letters \(\r_m\) corresponding to the lower central series \(\gamma_m(R)\) of the group of relations and leads to some neat computations.
\end{abstract}

\end{frontmatter}

\tableofcontents

\section{Introduction}

In a series of works \cite{emmanouil_group_2008, mikhailov_limits_2010, ivanov_higher_2015, fr_codes_2017}, the authors initiated the development of what can be called \textit{a functorial group theory}.
Let \(\C\) be a category of some sort of algebraic objects (e.g. groups, associative algebras, Lie algebras\ldots), and let \(A \in \C\) be some fixed object. The category of free presentations of \(A\), denoted \(\Pres(A)\), naturally arises in this context and leads to several fundamental questions.

Of particular importance is the investigation of which invariants of \(A\) can be obtained through the study of functors on the category \(\Pres(A)\). It is known for a functor \(\mathcal{F} \colon \Pres(A) \to \Set\) that \(\lim \, \mathcal{F} = \inv \, \mathcal{F}\), where \(\inv\) denotes a maximal constant subfunctor. For example, for a group \(G\) and a free extension \(R \mono F \epi G\) there is a Hopf's formula
\[H_2(G) \cong \frac{R \cap [F, F]}{[R, F]}.\]
One may verify \(H_2(G)\) is the maximal constant subfunctor of \(R/[R, F]\) considered as a functor \(\Pres(G) \to \Grp\). Namely, one has
\begin{equation}
\label{eq:hopf_functorial}
\lim \, R/[R, F] = H_2(G).
\end{equation}

The generalization of this formula to higher homology groups presents significant technical challenges. These obstacles can be resolved, however, by replacing the consideration of functorial subgroups of a free group \(F\) with the study of functorial ideals in the corresponding group ring \(\Z[F]\). Namely, for every free extension \(R \mono F \epi G\) there are two-sided ideals \(\r \subset \f \subset \Z[F]\) given by
\begin{align*}
\r &= (R - 1)\Z[F], \\
\f &= (F - 1)\Z[F],
\end{align*}
and one may show that
\begin{alignat*}{2}
H_{2n}(G) &= \frac{\r^n \cap \f \r^{n-1} \f}{\r^n \f + \f \r^n} = \lim \, \frac{\f}{\r^n \f + \f \r^n}, &\quad& n \ge 2, \\
H_{2n+1}(G) &= \frac{\f\r^k \cap \r^k \f}{\r^{n+1} + \f\r^n\f} = \lim \, \frac{\f}{\r^{n+1} + \f\r^n\f}, &\quad& n \ge 0.
\end{alignat*}
From the basic homological algebra we recall that the functor of limit
\[\lim \colon \Fun(\Pres(G), \Ab) \to \Ab\]
is left exact. It is standard that the aforementioned functor category admits sufficient injective objects, thereby guaranteeing the existence of right derived functors for computing limits. Following standard conventions, we denote those functors by
\[\lim^i \colon \Fun(\Pres(G), \Ab) \to \Ab, \quad i \ge 0,\]
reffering to them as \textit{higher limits}.
One may readily verify that for a functorial ideal \(\a \subset \Z[F]\) one has \(\lim \, \f/\a = \lim^1 \, \a\).
The utilization of higher limits consequently enables us to focus our analysis on subideals of the functorial enveloping ring \(\Z[F]\). One may form arbitrary sums and intersections of monomials in letters \(\f\), \(\r\) --- such as \(\r\f \cap \f\r\) or \(\r^{n+1} + \f\r^n\) --- and regard them as functors \(\Pres(G) \to \Ab\). We call such functors \textit{\(\f\r\)-codes}, and the fundamental problem is to calculate their higher limits.

As initially demonstrated in \cite{fr_codes_2017, standard_complex_2020}, a plethora of intriguing phenomena emerge already in this case.
Many non-trivial functors can be obtained in this way. For example,
\[\lim^1 \, \r\f\r + \f\r\f = \coker\{H_3(G) \otimes G_\ab \to H_2(G; \g \otimes_{\Z[G]} \g)\},\]
where \(\g = \Delta(G) \subset \Z[G]\) is a fundamental ideal of a group \(G\), and
\[\lim^2 \, \r\f\r + \f\r\r = (\g \otimes_{\Z[G]} \g) \otimes \g.\]
\textit{A functorial linguistics} is a study of functors obtained from basic `alphabetic' set of elementary functors (\textit{letters}) via some universal operations (such as sums, products and intersections). We think of higher limits of such functors as of the translation of the corresponding codes. Typically, this is a study of languages with very simple syntax but complex semantics.

It is important to emphasize that certain invariants still elude description within the \(\f\r\)-language paradigm. Apparently, the translations of \(\f\r\)-codes are limited to include homology only with integral coefficients. In the present work, we extend the language with an infinite family of new letters
\[\f \supset \r = \r_1 \supset \r_2 \supset \r_3 \supset \ldots,\]
given by the lower central series of the group of relations $R$:
\[\r_m = (\gamma_m(R) - 1)\Z[F].\]
We call this extension \(\f\r_\infty\)-language.
We are able to give a functorial description of some homology groups with finite coefficients in terms of higher limits of some sentences in \(\f\r_\infty\)-language.

We wish to highlight an alternative perspective on functorial linguistics. Specifically, numerous computations in combinatorial group theory inherently depend upon the selection of a free group presentation; see \cite{gruenberg_1960, stohr_homology_1993, hartley_homology_1991, hartley_note_1991, kuzmin_fnn_1987, kovach_homology_1992}. The application of higher limits limit enables the extraction of invariant data. Moreover, it helps to streamline the exposition: it turns out that many formulas may be written down more systematically in a uniform way. While certain computations of higher derived limits can be readily obtained from classical results through elementary diagram chasing, others require substantially more sophisticated techniques - an investment we contend is justified given their theoretical significance. 

\subsection{Review of main results}

In \cite[Theorem~4.17]{fr_codes_2017}, the authors establish a formula for computing higher derived limits of coinvariant modules arising from tensor powers of relation modules. In more detail, they prove
\[\lim^i \, R_\ab^{\otimes n} \otimes_{\Z[G]} M = H_{2n-i}(G; M), \quad 0 \le i < n,\]
for any group \(G\) and a \(G\)-module \(M\). We derive analogous formulas for the case of exterior powers in place of tensor powers; see Corollary~\ref{cor:basic_lim_ext_rel}, Proposition~\ref{prop:lim_i_ex_rel_modulo_torsion} and Proposition~\ref{prop:lim_ex_mod_p}.
It turns out that when \(\Ex^n R_\ab\) is used, torsion comes into play. For example, in Corollary~\ref{cor:basic_lim_ext_rel}, we prove for groups without torsion up to \(n\) and \(\Z\)-free \(\ZG\)-module the formula
\[\lim^i \, \Ex^n R_\ab \otimes_{\Z[G]} M = H_{n-i}(G; S^n(\g) \otimes M), \quad i = 0, 1,\]
where \(S^n(\g)\) is an \(n\)-th symmetric power of the augmentation ideal \(\g\) of the integral group ring \(\ZG\).

Next, we present our specific computations of some (higher) limits. They predominantly stem from classical results of Kuz'min, Kovach, St\"or, Hartley.

First, we compute the limits of the functors of the form \(H_n(F/\gamma_m(R))\). For example, for \(m = 2\) in Proposition~\ref{prop:hn_localized_away_2_kuzmin} we establish that there is an isomorphism over \(\Z[1/2]\)
\[\lim \, H_n(F/\gamma_2(R)) \otimes \Z[1/2] = \bigoplus_p f^{(p)}_n H_n(G; \Z/p\Z),\]
for a group \(G\) without \(n!\)-torsion. Here, \(f^{(p)}_n\) denotes the so-called Kuz'min polynomial; see \S~\ref{subsec:kuzmin} for details.

Second, we generalize the Hopf-type formula \eqref{eq:hopf_functorial} by proving an isomorphism
\[\lim \, \frac{[\overbrace{R, R, \ldots, R}^p]}{[\underbrace{R, R, \ldots, R}_p, F]} = H_4(G; \Z/p)\]
for a prime \(p\) and a group \(G\) without \(p\)-torsion; see Corollary~\ref{cor:lim_h2_lie}~(1).
This was known for \(p = 2\) and \(p = 3\); see \cite{ivanov_higher_2015}.

Finally, we provide the calculations for the aforementioned \(\f\r_\infty\)-language. To accomplish this, we establish several fundamental properties of higher limits (see e.g. Theorem~\ref{thm:lim_free_module}, Proposition~\ref{prop:lim_of_product}).
We compute the basic codes in the localized version of \(\f\r_\infty\)-language. See Section~\ref{sec:fr_inf} for an exposition of related results. We present selected calculations in explicit form.
\begin{itemize}
\item For a group \(G\) without \(p\)-torsion, we obtain
\[\lim^1 \, \r_p \f + \f \r_p = H_4(G; \Z/p).\]
\item For a group \(G\) without \(2\)-torsion, we obtain
\[\lim^1 \, \r_4 \f + \f \r_4 = H_6(G; \Z/p).\]
\item More interesting formulas can be obtained if we localize coefficients. For example, we prove that after localizing in \((p)\) for some odd \(p\), there is a formula
\[\lim^1_{\Z_{(p)}} \, \r_2^{(p^2+1)/2} + \f\r_2^{(p^2-1)/2}\f = H_{p^2+2}(G;\Z/p) \oplus H_{p^2+2p}(G;\Z/p)\]
for every \(p!\)-torsionless group.
\end{itemize}

\medskip
\textbf{Acknowledgment.}
We would like to thank Vladimir Sosnilo and Roman Mikhailov for their helpful suggestions.

The study of N.~Golub was carried out with the financial support of the Ministry of Science and Higher Education of the Russian Federation in the framework of a scientific project under agreement \texttt{075-15-2025-013}.

The work of V.~Ionin was supported by the Ministry of Science and Higher Education of the Russian Federation (agreement \texttt{075-15-2025-344} dated 29/04/2025 for Saint Petersburg Leonhard Euler International Mathematical Institute at PDMI RAS) and by the Theoretical Physics and Mathematics Advancement Foundation ``BASIS'', project no.~\texttt{24-7-1-26-3}.

\section{Preliminaries}

We devote this section to technical results and facts that we use extensively throughout the paper. Given the technical nature of this development, readers primarily concerned with applications may proceed directly to Section~\ref{sec:shifting}.

In this section \(\BaseRing\) denotes a commutative ring.

\subsection{Notable functors and sequences}

We follow the notations from \cite{stohr_homology_1993, hartley_homology_1991}.
Let \(M\) be a left \(\BaseRing\)-module.
\begin{enumerate}
\item (\textit{Tensor power}) For every \(n \ge 0\) there is a \(\BaseRing\)-module
\[T^n M = \underbrace{M \otimes_\BaseRing \ldots \otimes_\BaseRing M}_n.\]
The direct sum \(TM = \bigoplus_{n \ge 0} T^n M\) is an algebra in an evident way.
\item (\textit{Symmetric power}) For every \(n \ge 0\) there is a \(\BaseRing\)-module
\[S^n M = T^n M / \{(\ldots \otimes x_i \otimes \ldots\otimes  x_j \otimes \ldots) - (\ldots \otimes x_j \otimes \ldots \otimes x_i \otimes \ldots)\}.\]
If \(M\) is free with ordered basis \(\{x_i\}_{i \in I}\), then \(S^n M\) is free with the basis
\[\{x_{i_1} \circ \ldots \circ x_{i_n} \mid x_{i_1} \le \ldots \le x_{i_n}\}.\]
The composition
\[S^n M \xrightarrow{i} T^n M \epi S^n M\]
is the multiplication by \(n!\), where
\[i(x_{i_1} \circ \ldots \circ x_{i_n}) = \sum_{\sigma \in S_n} x_{i_{\sigma(1)}} \otimes \ldots \otimes x_{i_{\sigma(n)}}.\]
\item (\textit{Exterior power}) For every \(n \ge 0\) there is a \(\BaseRing\)-module
\[\Ex^n M = T^n M / \{(\ldots \otimes x_i \otimes \ldots\otimes  x_j \otimes \ldots) + (\ldots \otimes x_j \otimes \ldots \otimes x_i \otimes \ldots)\}.\]
If \(M\) is free with ordered basis \(\{x_i\}_{i \in I}\), then \(S^n M\) is free with the basis
\[\{x_{i_1} \land \ldots \land x_{i_n} \mid x_{i_1} < \ldots < x_{i_n}\}.\]
The composition
\[\Ex^n M \xrightarrow{i} T^n M \epi \Ex^n M\]
is the multiplication by \(n!\), where
\[i(x_{i_1} \land \ldots \land x_{i_n}) = \sum_{\sigma \in S_n} (-1)^\sigma x_{i_{\sigma(1)}} \otimes \ldots \otimes x_{i_{\sigma(n)}}.\]
\item (\textit{Divided power}) For every \(n \ge 0\) there is a \(\BaseRing\)-module \(\Gamma^n M\). It can be defined for a free module \(M\) as the submodule of \(S_n\)-invariants in the tensor power \(T^n M\).
\item (\textit{Lie power}) For every \(n \ge 0\) there is a \(\BaseRing\)-module \(\Lie^n M\) as the submodule of \(T^n M\) given as the intersection \(T^n M \cap \langle M \rangle_{\mathrm{Lie}}\), where \(\langle M \rangle_{\mathrm{Lie}}\) denotes a Lie subalgebra of \(TM^{\mathrm{Lie}}\)  generated by \(T^1 M = M\).
The direct sum
\[\Lie M = \bigoplus_{n \ge 0} \Lie^n M\]
is a graded Lie algebra over \(\BaseRing\).
\end{enumerate}

Given a group \(G\) there is a graded Lie ring
\[\bigoplus_{n \ge 1} \frac{\gamma_n(G)}{\gamma_{n+1}(G)}\]
with the bracket given via the commutator in \(G\).
Using the universal property of free Lie rings one defines a map 
\begin{equation}\label{eq:lie_of_group}
\Lie G_\ab \to \bigoplus_{n \ge 1} \frac{\gamma_n(G)}{\gamma_{n+1}(G)}.
\end{equation}

The following theorem is well-known.
\begin{fact}[{Magnus--Witt theorem; see e.g. \cite[Theorem~1]{ellis_2002}}]\label{fact:magnus_witt}
If \(G = F\) is a free group, then the map~\eqref{eq:lie_of_group} is an isomorphism.
\end{fact}
Further, as it is done in \cite{stohr_homology_1993, hartley_homology_1991}, we shall identify free Lie powers \(\Lie^n F_\ab\) with the quotients \(\gamma_n(F)/\gamma_{n+1}(F)\).

The functors \(\Ex^n, S^n\) and \(\Gamma^n\) are connected via the Koszul sequences.
\begin{fact}[{Koszul sequences; see e.g.~\cite[(1.4) and (1.5)]{breen_functorial_1999}}]
Let \(A \mono B \epi C\) be a short exact sequence of flat \(\BaseRing\)-modules. Then for any \(n \ge 1\) there are long exact sequences
\begin{align*}
&\Ex^n A \mono \Ex^{n-1} A \otimes_\BaseRing B \to \Ex^{n-2} A \otimes_\BaseRing S^2 B \to \cdots \to S^{n-1} B \otimes_\BaseRing A \to S^n B \epi S^n C, \\    
&\Gamma^n A \mono \Gamma^{n-1} A \otimes_\BaseRing B \to \Gamma^{n-2} A \otimes_\BaseRing \Ex^2 B \to \cdots \to \Ex^{n-1} B \otimes_\BaseRing A \to \Ex^n B \epi \Ex^n C.
\end{align*}
\end{fact}

\begin{prop}\label{prop:tor_free_group}
Let \(F\) be a free group, \(M\) be a left \(\BaseRing[F]\)-module, and \(N\) be a right \(\BaseRing[F]\)-module.
Then there are short exact sequences
\[\Tor^\BaseRing_n(M, N)_F \mono \Tor^{\BaseRing[F]}_n(M, N) \epi H_1(F; \Tor^\BaseRing_{n-1}(M, N)), \quad n \ge 1.\]
\end{prop}
\begin{proof}
Let \(M\) be a left \(\BaseRing[F]\)-module. Consider the composition of functors
\[\Mod(\BaseRing[F]) \xrightarrow{- \otimes M} \Mod(\BaseRing[F]) \xrightarrow{(-)_F} \kMod.\]
Then for every right \(\BaseRing[F]\)-module \(N\) there is the first quadrant homology spectral sequence of composition
\[\Tor^{\BaseRing[F]}_p(\BaseRing, \Tor^\BaseRing_q(M, N)) \simeq H_p(F; \Tor^\BaseRing_q(M, N)) \Rightarrow \Tor^{\BaseRing[F]}_{p+q}(M, N).\]
Since the free group \(F\) has a homological dimension one, we conclude the result.
\end{proof}

\subsection{Functorial modules}

Let \(\Mod\) denotes the category of pairs \((R, M)\), where \(R\) is a (possibly noncommutative) ring and \(M\) is a left \(R\)-module. Morphisms are pairs \((f, \varphi) \colon (R, M) \to (S, N)\), where \(f \colon R \to S\) is a ring homomorphims and \(\varphi \colon M \to N\) is a \(R\)-linear map (the abelian group \(N\) is a \(R\)-module with an action defined as \(rn = f(r)n\)). The category of right modules is defined analogously. There is a natural projection functor \(\Mod \to \Ring\).
Let \(\C\) be a small category and \(\O \colon \C \to \Ring\) be a functorial ring. Then the \textit{left \(\O\)-module \(\mathcal{F} \colon \C \to \Mod\)} is a functor, such that the following diagram commutes:
\[\begin{tikzcd}[cramped]
	&& \Mod \\
	\\
	\C && {\Ring.}
	\arrow[from=1-3, to=3-3]
	\arrow["{\mathcal{F}}", from=3-1, to=1-3]
	\arrow["\O"', from=3-1, to=3-3]
\end{tikzcd}\]
Let \(\FMod(\O)\) be a full subcategory of the functor category \(\Fun(\C, \Mod)\) spanned by all left \(\O\)-modules.
\begin{prop}
For every functorial ring \(\O\), the category \(\FMod(\O)\) is a Grothendieck category.
\end{prop}
\begin{proof}
Consider a ring without a unit \(\O[\C]\) defined as follows. Its elements are the functions \(u \colon \Mor(\C) \to \sqcup_{c \in \C} \O(c)\) such that \(u(\varphi) \in \O(\cod \, \varphi)\) for every \(\varphi \in \Mor(\C)\) and \(u(\varphi)\) is non-zero only for finitely many \(\varphi\)'s.
The sum and the product in \(\O[\C]\) is defined by the formulas
\begin{align*}
(u + v)(\varphi) &= u(\varphi) + v(\varphi), \\
(u \cdot v)(\varphi) &= \sum_{\varphi = \beta \circ \alpha} \O(\beta)(u(\alpha)) \cdot v(\beta).
\end{align*}
For every \(\varphi_0 \in \Mor(\C)\) there is a `characteristic element' \(\chi(\varphi_0) \in \O[\C]\):
\[\chi(\varphi_0)(\varphi) = \begin{cases}
1, & \varphi = \varphi_0, \\
0, & \varphi \ne \varphi_0.
\end{cases}\]
Given a \(\O\)-module \(\F\), one may construct a \(\O[\C]\)-module \(M = \oplus_{c \in \C} \F(c)\) with the scalar action
\(u \cdot (x_c)_{c \in \C} = (\sum_{\varphi \colon c' \to c} u(\varphi) \cdot \mathcal{F}(\varphi)(x_{c'}))_{c \in \C}\). On the other hand, for every \(\O[\C]\)-module \(M\) there is an obvious functorial \(\O\)-module \(\F\) given by \(\F(c) = \chi(\id_c) \cdot M\) and \(\F(\varphi)\colon m \mapsto \chi(\varphi) m\).

It is straightforward to verify that the correspondence above establishes an adjoint pair, which realizes \(\FMod(\O)\) as a reflective subcategory of the Grothendieck category \(\Mod(\O[\C])\). The reflector \(M \mapsto \F\) clearly preserves finite limits, and the result follows.
\end{proof}

It is important to have some convenient framework that controls the behaviour of functorial modules under changing the base ring. The correspondence \(R \rightsquigarrow \Mod(R)\) depends on the \(R\) pseudofunctorially. The modern way to handle pseudofunctors is by means of Grothendieck (op)fibrations. The main idea is that instead of considering a pseudofunctor \(\B \to \mathsf{Cat}\), we can study an ordinary functor \(\F \colon \A \to \B\) such that the fibers \(\F^{-1}(b)\) depend pseudofunctorially on \(b \in \B\).

Recall that a functor \(\F \colon \A \to \B\) is a \textit{fibration} if for every \(a \in \A\) and \(f \colon b \to \F(a)\) there exists a cartesian lift of \(f\) with codomain \(a\):
\[\begin{tikzcd}[cramped]
	\smallbullet \\
	& {f^*a} && a && \A \\
	\smallbullet \\
	& b && {\F(a)} && \B
	\arrow["{{\exists!\,\widetilde w}}"{description}, dashed, from=1-1, to=2-2]
	\arrow["{{\forall\, g}}", curve={height=-12pt}, from=1-1, to=2-4]
	\arrow[dotted, maps to, from=1-1, to=3-1]
	\arrow["{{\text{cartesian}}}"', from=2-2, to=2-4]
	\arrow[dotted, maps to, from=2-2, to=4-2]
	\arrow[dotted, maps to, from=2-4, to=4-4]
	\arrow["{\F}", from=2-6, to=4-6]
	\arrow["{\forall \, w}"{description}, from=3-1, to=4-2]
	\arrow[curve={height=-12pt}, from=3-1, to=4-4]
	\arrow["f"', from=4-2, to=4-4]
\end{tikzcd}\]
A functor \(\F\) is an \textit{opfibration} if \(\F^\mathrm{op}\) is a fibration.
\begin{ex}
A projection \(\Mod \to \Ring\) is a bifibration. Let \(f \colon R \to S\) be a ring homomoprism. The fibration structure is given by the restriction \(f^* \colon \Mod(S) \to \Mod(R)\) and the opfibration structure is given by the extension \((- \otimes_R S) \colon \Mod(R) \to \Mod(S)\).
\end{ex}
\begin{ex}
For every bifibration \(\A \to \B\) and index category \(\C\) the post-composition functor \(\A^\C \to \B^\C\) is a bifibration as well; see \cite[Lemma~4.7]{props_2023}. In particular, there is a bifibration \(\Mod^\C \to \Ring^\C\) such that the fiber over a functorial ring \(\O\colon \C \to \Ring\) is the Grothendieck category \(\FMod(\O)\).
\end{ex}

If \(R\) is some ring, and \(\O \equiv R\) is a constant functorial ring on some index category \(\C\), then there is an equivalence \(\FMod(\O) \simeq \Fun(\C, \Mod(R))\).
In this case, there is an obvious adjunction
\[\begin{tikzcd}[cramped]
	{\Fun(\C, \Mod(R))} && {\Mod(R),}
	\arrow[""{name=0, anchor=center, inner sep=0}, "\lim", shift left=2, from=1-1, to=1-3]
	\arrow[""{name=1, anchor=center, inner sep=0}, "\const", shift left=2, from=1-3, to=1-1]
	\arrow["\dashv"{anchor=center, rotate=90}, draw=none, from=1, to=0]
\end{tikzcd}\]
and for \(i \ge 0\) we can define higher limits \(\lim^i \colon \Fun(\C, \Mod(R)) \to \Mod(R)\) as the right derived functors \(\RR^i \lim\), using usual construction from homological algebra.

It is more convenient to switch to the homotopical perspective on the derived functors. Let \(\Delta\) be a category of non-empty finite linear posets and non-decreasing maps. Then for every category \(\A\) there is a category \(\A^\Delta\) of cosimplicial objects in \(\A\). If \(\A\) is abelian, then there is a Dold-Kan correspondence that establishes an equivalence between \(\A^\Delta\) and the category \(\mathsf{Ch}^{\ge 0}(\A)\) of non-negatively graded cochain complexes. Under this equivalence, the cohomotopy groups \(\pi^i\) of a cosimiplicial object correspond to the cohomology groups \(H^i\) of a cochain complex.
If \(\A\) is a Grothendieck category, then there exists a combinatorial model structure on \(\A^\Delta\) with weak equivalences being quasi-isomorphisms (e.g. one can choose standard injective or projective model structure).
Then there is an injective model structure on the functor category \(\Fun(\C, \A^\Delta)\).
An adjunction \((\const \dashv \lim)\) lifts to the Quillen adjunction
\[\begin{tikzcd}[cramped]
	{\Fun(\C, \A^\Delta)_\inj} && {\A^\Delta.}
	\arrow[""{name=0, anchor=center, inner sep=0}, "\lim", shift left=2, from=1-1, to=1-3]
	\arrow[""{name=1, anchor=center, inner sep=0}, "\const", shift left=2, from=1-3, to=1-1]
	\arrow["\dashv"{anchor=center, rotate=90}, draw=none, from=1, to=0]
\end{tikzcd}\]
There is a total right derived functor \(\RR \lim \colon \Fun(\C, \A^\Delta) \to \A^\Delta\), and for a functor \(M\colon \C \to \A\) there is a natural equivalence
\(\lim^i \, M \simeq \pi^i(\RR \lim \, M)\), where \(M\) on the right is considered as a copresheaf of constant cosimplicial objects \(c \mapsto \const(M)\).
The computational difficulty stems from the fact that to compute \(\RR \lim \, M\) we have to choose some fibrant replacement for \(M\) in \(\Fun(\C, \A^\Delta)_\inj\), and the fibrations in the injective model structure are not obvious. However, when the index category \(\C\) shares some nice categorical properties, namely when it is strongly connected and have binary coproducts, we can use the so-called standard complex to bypass this difficulty \ref{standardcomplex}.

\subsection{Main example}
Let \(G\) be a group and consider a category \(\C = \Pres(G)\) of free extensions of \(G\). The objects of \(\Pres(G)\) are short exact sequences of groups \(R \mono F \epi G\) with \(F\) being free, and the morphisms are the commutative diagrams of the form
\[\begin{tikzcd}[cramped]
	{R_1} & {F_1} & G \\
	{R_2} & {F_2} & {G.}
	\arrow[hook, from=1-1, to=1-2]
	\arrow[from=1-1, to=2-1]
	\arrow[two heads, from=1-2, to=1-3]
	\arrow[from=1-2, to=2-2]
	\arrow[equals, from=1-3, to=2-3]
	\arrow[hook, from=2-1, to=2-2]
	\arrow[two heads, from=2-2, to=2-3]
\end{tikzcd}\]
For every commutative ring \(\BaseRing\) there is a functorial ring \(\O \colon \Pres(G) \to \Ring\) given by
\[\O(R \mono F \epi G) = \BaseRing[F],\]
the free group \(\BaseRing\)-algebra.
We mainly interested in the study of higher limits of functorial \(\O\)-module \(M\).

It turns out that there is a plethora of interesting phenomena occuring when we specialize to the case \(M = \a\) for some functorial ideal \(\a \subset \BaseRing[F]\) (explicitely, this means that \(\a(R \mono F \epi G)\) is some two-sided ideal of \(\BaseRing[F]\) and the structure maps of the functor \(\a\) are restricted from those of \(\O = \BaseRing[-]\)).

There are functorial ideals \(\BaseRing[-] \supset \f \supset \r\) given by
\begin{align*}
\f(R \mono F \epi G) &= \ker\{\BaseRing[F] \to k\} = (F - 1)\BaseRing[F], \\
\r(R \mono F \epi G) &= \ker\{\BaseRing[F] \to \BaseRing[G]\} = (R - 1)\BaseRing[F].
\end{align*}
That is, \(\f = \Delta_\BaseRing(F)\) and \(\f/\r = \g = \Delta_\BaseRing(G)\).

\subsection{Relation module}

Let us remind the basics on the relation modules. One may use \cite{stohr_homology_1993} for a general reference.

Let \(\BaseRing = \Z\).
Given a free extension \(R \mono F \epi G\), a \textit{relation module} associated with it may be identified with the kernel in the short exact sequence
\begin{equation}\label{eq:relation_module}
R_\ab \mono P \epi \g,
\end{equation}
where \(P = \ZG \otimes_\ZF \f\).
The sequence \eqref{eq:relation_module} is natural, that is, it is a short exact sequence in \(\Fun(\Pres(G), \Mod(\ZG))\), and is referred to as a \textit{relation sequence}. The embedding of the relation module \(R_\ab\) into a free \(\ZG\)-module \(P\) is called \textit{Magnus embedding}.

In terms of \(\f\r\)-language, the sequence \eqref{eq:relation_module} can be compactly written as an obvious extension \(\r/\r\f \mono \f/\r\f \epi \f/\r\).

\subsection{Computational toolkit}
In this section, we give the basic computational tools for higher limits: standard complex, limits spectral sequence, and K\"unneth formula.

\subsubsection{Standard complex}\label{standardcomplex}

If \(\C\) is a strongly connected category with binary coproducts, then for every \(c \in \C\) there is a cosimplicial object \(\BB \colon \Delta \to \C\) given by \(\BB(c)^n = \sqcup^{n+1} c\) with the natural structure maps.
In \cite[Theorem~2.12]{standard_complex_2020}, the authors prove that for a functor \(\F \colon \C \to \Ab\) one has a natural isomorphism
\[\lim^n \, \F = \pi^n \F \BB(c).\]
As always, the cohomotopy groups of a cosimplicial abelian group \(A\) can be computed with cohomology groups of the Moore (cochain) complex \(QA\) or the alternating sum complex \(CA\). The natural projection \(CA \epi QA\) is a homotopy equivalence.

\begin{prop}
Let \(f \colon \BaseRing \to \BaseRing'\) be a homomorphism of commutative rings. Let \(\C\) be a strongly connected category with binary coproducts.
\begin{enumerate}
\item If \(M\) is a functorial \(\BaseRing'\)-module on \(\C\), then the restriction of scalars induces a natural isomorphism
\[\lim^\smallbullet_{\BaseRing'} \, M \simeq \lim^\smallbullet_\BaseRing\, M.\]
\item If \(f\) is flat, and \(M\) is a functorial \(\BaseRing\)-module on \(\C\), then the extension of scalars induces a natural isomorphism
\[\lim^\smallbullet_{\BaseRing'} \, (M \otimes_\BaseRing \BaseRing') \simeq (\lim^\smallbullet_\BaseRing \, M) \otimes_\BaseRing \BaseRing'.\]
\end{enumerate}
\end{prop}
\begin{proof}
Recall that the restriction of scalars \(f^* \colon \Mod(\BaseRing') \to \Mod(\BaseRing)\) is an exact functor, and the extension of scalars \((- \otimes_\BaseRing \BaseRing') \colon \Mod(\BaseRing) \to \Mod(\BaseRing')\) is exact precisely when \(f\) is flat.

The result follows since higher limits could be computed as cohomologies of the standard complex, and exact functors commute with cohomologies.
\end{proof}

\subsubsection{Spectral sequence}

\begin{fact}[{\cite[Corollary~3.17]{fr_codes_2017}}]
Let \(\mathcal{F}^\smallbullet\) be a bounded below cochain complex of representations of a category \(\C\) with \(\lim\)-acyclic cohomologies. Then there exists a cohomological spectral sequence
\[E_1^{p,q} = \lim^q\, \mathcal{F}^p \Rightarrow \lim\, H^n(\mathcal{F}^\smallbullet).\]
\end{fact}

The subsequent propositions follow by a simple chasing using the above spectral sequence.
\begin{prop}\label{prop:spectral_application_1}\leavevmode
Let \(a_0 \mono a_1 \to a_2 \epi a_3\) be an exact sequence of representations. If \(n \ge 1\) is such that
\[\lim^1 \, a_0 = \lim^2 \, a_0 = \ldots = \lim^n \, a_0 = 0,\]
then there is an exact sequence
\[\lim\, a_0 \mono \lim \, a_1 \to \lim\, a_2 \to \lim\, a_3 \to \lim^1\,a_1 \to \lim^1\, a_2 \to \lim^1 \, a_3 \to \cdots \to \lim^n \, a_2.\]
\end{prop}

\begin{prop}\label{prop:spectral_application_2}\leavevmode
Let \(a_0 \mono a_1 \to a_2 \epi a_3\) be an exact sequence of representations.
\begin{enumerate}
\item If \(\lim \, a_2 = 0\), then \(\lim\, a_0 = \lim\, a_1\).
\item If \(n \ge 1\) is such that
\[\lim\, a_2 = \lim^1\,a_2 = \ldots = \lim^n\,a_2 = 0,\]
then there is an exact sequence
\[\lim^1 \, a_0 \mono \lim^1 \, a_1 \to \lim \, a_3 \to \lim^2 \, a_0 \to \lim^2 \, a_1 \to \lim^1 \, a_3 \to \cdots \to \lim^n a_1.\]
\end{enumerate}
\end{prop}

\subsection{Monoadditiviy and split monoadditivity}
Let \(\C\) be a strongly connected category with binary coproducts. There are notions of \textit{monoadditive} and \textit{split monoadditive} representations \(M \colon \C \to \Mod(\BaseRing)\) defined in \cite{fr_codes_2017}. These notions allow us to quickly obtain a vanishing of some limits. We recall just basic facts about it. These conceptual tools facilitate determining the vanishing conditions for certain limits. For completeness, we briefly review the fundamental properties underlying this result.
\begin{fact}[{\cite[Corollary~3.14, \S 4.2]{fr_codes_2017}}]\leavevmode
\begin{enumerate}
\item A representation \(M\) is monoadditive if and only if \(\lim \, M = 0\).
\item If \(M\) is split monoadditive, then all limits \(\lim^* \, M\) vanish.
\item For \(\C = \Pres(G)\) the representations of the form \(\f \otimes_{\BaseRing[F]} M\) and \(M \otimes_{\BaseRing[F]} \f\) are split monoadditive.
\end{enumerate}
\end{fact}

\subsubsection{K\"unneth formula}

Let \(\BaseRing\) be a commutative ring. A \textit{functorial \(\BaseRing\)-algebra} is a functorial ring \(\O\) equipped with a natural transformation \(\const(\BaseRing) \to \O\). The following proposition is well known.
\begin{prop}
Let \(\O\) be a functorial \(\BaseRing\)-algebra.
\begin{enumerate}
\item The limits \(\lim^\smallbullet_\BaseRing \, \O\) form a \(\Z_{\ge 0}\)-graded \(k\)-algebra in a natural way.
\item If \(M\) is an \(\O\)-module, then the limits \(\lim^\smallbullet_\BaseRing \, M\) form a graded \((\lim^\smallbullet \, \O)\)-algebra.
\end{enumerate}
\end{prop}

\begin{defn}
We say that a representation \(M\) of a category \(\C\) is \textit{finitary} if there are only finitely many non-zero higher limits \(\lim^\smallbullet \, M\).
\end{defn}

\begin{cor}\label{cor:kunneth_appl}
Let \(\BaseRing\) be a heriditary commutative ring. Let \(M\) and \(N\) be finitary representations of \(\BaseRing[F]\)-modules such that one of them is flat. Then for every \(n \ge 0\) there is an exact sequence
\[\bigoplus_{i+j=n} \lim^i\, M \otimes_\BaseRing \lim^j\,N \mono \lim^n \, M \otimes_{\BaseRing[F]} N \epi \bigoplus_{i+j=n+1} \Tor^\BaseRing_1(\lim^i\, M, \lim^j\,N).\]
\end{cor}
\begin{proof}
Note that \(\lim^\smallbullet \, \BaseRing[F] = (\BaseRing, 0, 0, \ldots)\), and since \(\BaseRing\) is heriditary, it has global dimension \(1\), and the K\"unneth spectral sequence \cite[Section~3]{standard_complex_2020} degenerates to the stated short exact sequences.
\end{proof}

\section{Shifting formulas}
\label{sec:shifting}

In this section, we establish shifting formulas that eventually enable partial computation of the higher limits of the functors \((\Ex^n R_\ab)_G\); see Corollary~\ref{cor:basic_lim_ext_rel}, Proposition~\ref{prop:lim_i_ex_rel_modulo_torsion} and Proposition~\ref{prop:lim_ex_mod_p}. Our method is a direct generalization of the method suggested in \cite[Proposition~6.1]{ivanov_higher_2015}.

We say that a group \(G\) is \textit{\(n\)-torsionless}, if no nontrivial element \(x\) of \(G\) satisfies the equation \(x^n = 1\). This is equivalent to the statement that there is no element in \(G\) of order \(p\) for every prime divisor \(p\) of \(n\).
Note that a group \(G\) is \(n!\)-torsionless if and only if there is no element in \(G\) of order \(p\) for every prime \(p\) such that \(p \le n\).

By employing a standard dimension-shifting argument, we obtain a canonical isomorphism that elucidates the fundamental relationship between homology groups with exterior power coefficients and their symmetric power counterparts.
This result generalizes \cite[Proposition~1.1]{hartley_note_1991}.
\begin{lem}\label{lem:hmlgy_ex_sym_shift}
Let \(n \ge 1\) and \(G\) be a \(n!\)-torsionless group. Let \(A \mono B \epi C\) be a short exact sequence of \(\ZG\)-modules, such that \(B\) is free as a \(\ZG\)-module and \(C\) is a torsionless abelian group. Then for every \(\ZG\)-module \(M\) there are isomorphisms
\[H_{i+n}(G; S^n C \otimes M) \simeq H_i(G; \Ex^n A \otimes M), \quad i \ge 1.\]
\end{lem}
\begin{proof}
Since \(A, B\) and \(C\) are torsionless abelian groups, there is a corresponding Koszul sequence.
Tensoring it over \(\Z\) with \(M\) yields a long exact sequence
\begin{equation}
\label{eq:tensored_koszul}
\Ex^n A \otimes M \mono \Ex^{n-1} A \otimes B \otimes M \to \dots \to S^n B \otimes M \epi S^n C \otimes M.
\end{equation}

Since \(G\) has no torsion up to \(n\) we conclude that for any free \(\ZG\)-module \(B\) the exterior powers \(S^k B, k \le n\), are free \(\ZG\)-modules; see \cite[Lemma~4.1]{hartley_homology_1991}.
Since \(A\) is a free abelian group, the exterior powers \(\Ex^* A\) are free abelian groups, and all modules \(\Ex^{n-k} A \otimes S^k B\) are free as \(\ZG\)-modules; see \cite[III \S 5, Corollary~5.7]{brown_cohomology_1982}. Hence, by Shapiro's lemma, tensoring the resulting sequence with any \(M\) yields modules with trivial homology, i.e.
\[H_i(G; \Ex^{n-k} A \otimes S^k B \otimes M) = 0, \quad i \ge 1.\]

Finally, by a simple dimension shifting we derive the result.
\end{proof}

The following proposition establishes that, within the functorial framework and under the specified additional hypotheses, the higher limits induce inverse shiftings relative to those presented in the preceding lemma.
\begin{prop}\label{prop:lim_ex_sym_shift}
Let \(G\) be a \(n!\)-torsionless group and \(M\) be a \(\ZG\)-module. If for some \(0 \le k \le n\) we have
\[\lim^i \, (\Ex^{n-j} R_\ab \otimes S^j P \otimes M)_G = 0, \quad 0 \le i \le k, \; 1 \le j \le k - i + 1,\]
then there are isomorphisms
\[\lim^i\, (\Ex^n R_\ab \otimes M)_G = H_{n-i}(G; S^n(\g) \otimes M), \quad 0 \le i \le k.\]
If \(k = n\), then \(\lim^i \, (\Ex^n R_\ab \otimes M)_G = 0\) for \(i > n\).
\end{prop}
\begin{proof}
Consider a Koszul sequence for the relation sequence \eqref{eq:relation_module}.
Cut it into short exact sequences by introducing
\[N_j = \coker \{\Ex^{n-j+1} R_\ab \otimes S^{j-1} P \to \Ex^{n-j} R_\ab \otimes S^j P\}, \quad 1 \le j \le n.\]
For convenience, we set \(N_0 = \Ex^n R_\ab\). Note that \(N_n = S^n(\g)\).
Then we can cut the long exact sequence~\eqref{eq:tensored_koszul} into short exact sequences
\[N_{j-1} \otimes M \mono \Ex^{n-j} R_\ab \otimes S^j P \otimes M \epi N_j \otimes M, \quad 1 \le j \le n.\]
The middle term is \(H_*(G; -)\)-acyclic (see the proof of the previous lemma), hence for every \(1 \le j \le n\) there are isomorphisms
\begin{equation}\label{eq:ni_shift}
H_{l+1}(G; N_j \otimes M) \simeq H_l(G; N_{j-1} \otimes M), \quad l \ge 1,
\end{equation}
and there is the short exact sequence
\begin{equation}\label{eq:ni_4_term}
H_1(G; N_j \otimes M) \mono (N_{j-1} \otimes M)_G \to (\Ex^{n-j} R_\ab \otimes S^j B \otimes M)_G \epi (N_j \otimes M)_G.
\end{equation}
Iterating \eqref{eq:ni_shift}, we obtain for any \(0 \le j \le n\) that
\begin{equation}\label{eq:ni_h1_identification}
H_1(G; N_j \otimes M) \simeq H_{n-j+1}(G; N_n \otimes M) = H_{n-j+1}(G; S^n(\g) \otimes M)
\end{equation}
is a constant functor, which implies that it has trivial higher limits.

Since functor \(\lim^0\) is left exact, applying it to the sequence \eqref{eq:ni_4_term} yields identifications
\begin{equation}\label{eq:ni_homology_as_lim}
\lim \, (N_{j-1} \otimes M)_G = H_1(G; N_j \otimes M) \overset{\eqref{eq:ni_h1_identification}}{=} H_{n-j+1}(G; S^n(\g) \otimes M)
\end{equation}
for \(1 \le j \le k\).

Iteratively applying Proposition~\ref{prop:spectral_application_1} to the 4-term exact sequence \eqref{eq:ni_4_term}, we obtain the sequence of identifications
\begin{align*}
H_{n-i}(G; S^nC \otimes M) &\overset{\eqref{eq:ni_homology_as_lim}}{=} \lim \, (N_i \otimes M)_G = \lim^1\, (N_{i-1} \otimes M)_G = \cdots \\
&= \lim^i \, (N_0 \otimes M)_G = \lim^i \, (\Ex^n R_\ab \otimes M)_G  
\end{align*}
for \(0 \le i \le k\).
\end{proof}

The following lemma provides techniques for verifying the vanishing of certain common limits.
\begin{lem}\label{lem:h0_m_o_p_monoadditive}
Let \(G\) be a group, and let \(M\) be some functorial \(\ZG\)-module.
\begin{enumerate}
\item The functor \(H_0(G; M \otimes P)\)
is split monoadditive.
\item If \(G\) is \(2\)-torsionless, then \(H_0(G; M \otimes \Ex^2(P))\) is monoadditive.
\item If \(M\) is free \(\ZG\)-module, then \(H_0(G; M \otimes R_\ab)\) is monoadditive.
\end{enumerate}
\end{lem}
\begin{proof}
(1) First part of the lemma holds, since
\[H_0(G; M \otimes P) = M \otimes_\ZG (\ZG \otimes_\ZF \f) = M \otimes_\ZF \f.\]

(2) Note that since \(G\) has no \(2\)-torsion, \(S^2(P)\) is a free \(\ZG\)-module; see \cite[Lemma~4.1]{hartley_homology_1991}.
Write the Koszul sequence for \(S^2(P)\) over \(\Z\):
\begin{equation}\label{eq:ex2p_flat}
\Ex^2(P) \mono P \otimes P \epi S^2(P).    
\end{equation}
Since it consists of flat \(\ZG\)-modules, we can tensor it over \(\ZG\) with \(M\) to obtain the exact sequence
\[(M \otimes \Ex^2(P))_G \mono (M \otimes P \otimes P)_G \epi (M \otimes S^2(P))_G.\]
Note that the middle term is split monoadditive by part (1). The result follows.

(3) Tensor the short exact sequence of abelian groups \eqref{eq:relation_module} with \(M\) over \(\ZG\) to obtain the exact sequence
\[(M \otimes R_\ab)_G \mono (M \otimes P)_G \epi (M \otimes \g)_G.\]
The middle term is split monoadditive by part (1). The result follows.
\end{proof}

The proof of the next proposition is technical and has been deferred to Appendix~\ref{app:1}.
\begin{prop}\label{prop:s2_p_monoadditive}
Let \(G\) be a \(2\)-torsionless group. Then for every \(n \ge 0\) and \(\Z\)-free functorial \(\ZG\)-module \(M\) the functor
\(H_0(G; \Ex^n R_\ab \otimes S^2(P) \otimes M)\) is monoadditive.
\end{prop}

We now recall an important result that underlies many of the findings presented in subsequent sections.
\begin{fact}[{see \cite[Theorem~5.1]{hartley_homology_1991}}]\label{fact:symmetric_power_homologies_torsion}
Let \(n \ge 2\), and let \(G\) be a group. Let \(M\) be a \(\Z\)-free \(\ZG\)-module. Then the homology groups \(H_k(G; S^n(\g) \otimes M)\) are periodic of exponent \(2n(n-1)\) for \(k \ge 1\), and \(H_0(G; S^n(\g) \otimes M)\) is a direct sum of a free abelian group and a periodic group of exponent \(2n(n-1)\).
\end{fact}

\begin{cor}\label{cor:basic_lim_ext_rel}
Let \(n \ge 2\), \(G\) be a \(n!\)-torsionless group, \(M\) be a \(\Z\)-free \(\ZG\)-module. Then one has
\[\lim^i \, (\Ex^n R_\ab \otimes M)_G = H_{n-i}(G; S^n(\g) \otimes M), \quad i = 0, 1,\]
and this group is torsion of exponent dividing \(2n(n-1)\).
\end{cor}
\begin{proof}
We want to verify the vanishing of the following limits:
\begin{align*}
&\lim \, (\Ex^{n - 1} R_\ab \otimes P \otimes M)_G = 0, \\
&\lim^1 \, (\Ex^{n - 1} R_\ab \otimes P \otimes M)_G = 0, \\
&\lim \, (\Ex^{n - 2} R_\ab \otimes S^2(P) \otimes M)_G = 0.
\end{align*}
First two limits vanish by Lemma~\ref{lem:h0_m_o_p_monoadditive} (1). The last limit vanishes by Proposition~\ref{prop:s2_p_monoadditive}.

The isomorphism follows from Proposition~\ref{prop:lim_ex_sym_shift}.
The exponent bounds follow from Fact~\ref{fact:symmetric_power_homologies_torsion}.
\end{proof}
\begin{prop}\label{prop:lim_i_ex_rel_modulo_torsion}
Let \(n \ge 2\), and let \(G\) be a group without \(n!\)-torsion, \(M\) be a \(\Z\)-free \(\ZG\)-module. Then one has for \(0 \le i \le n\):
\[\lim^i \, (\Ex^n R_\ab \otimes M)_G \otimes \Z[1/n!] = \begin{cases}
F((S^n(\g) \otimes M)_G) \otimes \Z[1/n!], & i = n, \\
0, & i \ne n,
\end{cases}\]
where \(F(A) = A/tA\) is a free part of an abelian group \(A\).
\end{prop}
\begin{proof}
We notice that when we work over \(\Z[1/n!]\), we have split monomorphisms
\[S^k(P) \otimes \Z[1/n!] \mono T^k(P), \quad 1 \le k \le n,\]
given by \((a_1 \circ \ldots \circ a_k) \mapsto \tfrac{1}{k!} \sum_{\sigma \in S_k} a_{\sigma(1)} \otimes \ldots \otimes a_{\sigma(k)}\). Hence, we have split monomorphisms
\[(\Ex^{n-k} R_\ab \otimes S^k(P) \otimes M)_G \otimes \Z[1/n!] \mono (\Ex^{n-k} R_\ab \otimes T^k(P) \otimes M)_G \otimes \Z[1/n!],\]
and the functor on the right is split monoadditive by Lemma~\ref{lem:h0_m_o_p_monoadditive} (1).
Hence, Proposition~\ref{prop:lim_ex_sym_shift} applies, and
\[\lim^i \, (\Ex^n R_\ab \otimes M)_G \otimes \Z[1/n!] = \begin{cases}
H_{n-i}(G; S^n(\g) \otimes M) \otimes \Z[1/n!], & i \le n, \\
0, & i > n.
\end{cases}\]
The result follows from Fact~\ref{fact:symmetric_power_homologies_torsion}.
\end{proof}

\begin{lem}\label{lem:lim_commutes_with_torsion}
Let \(\C\) be a category, and \(M \colon \C \to \Ab\) be some functor.
\begin{enumerate}
\item One has \(\lim \, tM = t(\lim \, M)\).
\item If \(\lim\, M \otimes \Q = 0\), then
\(\lim \, M\) is torsion.
\end{enumerate}
\end{lem}
\begin{proof}
Consider an exact sequence
\[tM \mono M \to M \otimes \Q.\]
The result follows, since the functor \(\lim\) is left exact and commutes with localizations.
\end{proof}

\begin{prop}\label{prop:lim_ex_mod_p}
Let \(p\) be a prime, and \(1 \le r \le p\) be an integer. Let \(G\) be a group without \((rp)!\)-torsion.
\begin{enumerate}
\item If \(r < p\) and \(i = 0, 1\), then
\[\lim^i \, (\Ex^{rp} R_\ab)_G \otimes \Z_{(p)} = H_{r(p+2)-i}(G; \Z/p\Z).\]
\item If \(r = p\), then
\[\lim \, (\Ex^{p^2} R_\ab)_G \otimes \Z_{(p)} = H_{p^2+2}(G; \Z/p\Z) \oplus H_{p^2+2p}(G; \Z/p\Z).\]
\end{enumerate}
\end{prop}
\begin{proof}
Note that \(\lim^i \, (\Ex^{rp} R_\ab)_G = H_{rp-i}(G; S^{rp}(\g))\) is a torsion group by the previous proposition. Next, for every abelian group \(A\) one has an equality \(t(A \otimes \Z_{(p)}) = t_p(A)\), where \(t(-)\) denotes the torsion part, and \(t_p(-)\) denotes its \(p\)-primary component. Hence, by Lemma~\ref{lem:lim_commutes_with_torsion} (1) we have
\[\lim^i \, (\Ex^{rp} R_\ab)_G \otimes \Z_{(p)} = t_p H_{rp-i}(G; S^{rp}(\g)).\]
These groups are computed in \cite[Proposition~3.2, Theorem~1]{hartley_note_1991}.
\end{proof}

\begin{ex}
In the following table we summarize our computations of \(\lim^i \, (\Ex^n R_\ab)_G\) for \(i = 0, 1\) and \(n \le 7\). The symbol \(t_k\) denotes an abelian group of exponent \(k\).
\begin{table}[H]
\resizebox{\textwidth}{!}{%
\begin{tabular}{|c|c|c|c|}
\hline
\(G\)               & \(\mathcal{F}\)     & \(\lim \, \mathcal{F}\)                         & \(\lim^1 \, \mathcal{F}\)              \\ \hline
any                 & \((R_\ab)_G\)       & \(H_2(G)\)                                      & \(H_1(G)\)                             \\ \hline
\(2\)-torsionless   & \((\Ex^2 R_\ab)_G\) & \(H_4(G; \Z/2)\)                                & \(H_3(G; \Z/2)\)                       \\ \hline
\(6\)-torsionless   & \((\Ex^3 R_\ab)_G\) & \(t_2 \oplus H_5(G; \Z/3)\)                     & \(t_2 \oplus H_4(G; \Z/3)\)            \\ \hline
\(24\)-torsionless   & \((\Ex^4 R_\ab)_G\) & \(t_3 \oplus H_6(G; \Z/2) \oplus H_8(G; \Z/2)\) & \(t_2 \oplus t_3\)                                      \\ \hline
\(120\)-torsionless  & \((\Ex^5 R_\ab)_G\) & \(t_2 \oplus H_7(G; \Z/5)\)                     & \(t_2 \oplus H_6(G; \Z/5)\)            \\ \hline
\(720\)-torsionless  & \((\Ex^6 R_\ab)_G\) & \(t_2 \oplus t_5 \oplus H_{10}(G; \Z/3)\)       & \(t_2 \oplus t_5 \oplus H_9(G; \Z/3)\) \\ \hline
\(5040\)-torsionless & \((\Ex^7 R_\ab)_G\) & \(t_2 \oplus t_3 \oplus H_9(G; \Z/7)\)          & \(t_2 \oplus t_3 \oplus H_8(G; \Z/7)\) \\ \hline
\end{tabular}%
}
\end{table}
\end{ex}
\section{On homology of \(F/\gamma_m(R)\)}
\label{sec:quotients}

This section examines invariant properties of the functors
\[H_n(F/\gamma_m R) \colon \Pres(G)\to \Ab.\]
Our primary interest in these properties stems from their role in defining invariants that admit $\f\r_{\infty}$-codes (see Section~\ref{sec:fr_inf}) and are closely related to homology of groups with cyclic coefficients, as we recall in this section.

We will see that the edge cases \(m = 2\) and \(n = 2\) are most approachable, so we will consider them separately.

\subsection{\(H_*(F/\gamma_2(R))\)}
\label{subsec:kuzmin}
\begin{fact}[{\cite[Theorem~2, Theorem~3]{kuzmin_fnn_1987}}]\label{fact:kuzmin_ex_rel_gamma_2}
Let \(n \ge 2\), and \(G\) be a \(n!\)-torsionless group, and \(R \mono F \epi G\) be some free extension of it.
\begin{enumerate}
\item The inclusion \(R_\ab \mono F/\gamma_2(R)\) induces a homomorphism
\[(\Ex^n R_\ab)_G \simeq H_n(R_\ab)_G \to H_n(F/\gamma_2(R)).\]
\item If \(p\) is some odd prime dividing \(n\), then this homomoprhism becomes isomorphism after tensoring with \(\Z_{(p)}\):
\[(\Ex^n R_\ab)_G \otimes \Z_{(p)} \xrightarrow{\simeq} H_n(F/\gamma_2(R)) \otimes \Z_{(p)}.\]
\item This homomorphism becomes isomorphism rationally:
\[(\Ex^n R_\ab)_G \otimes \Q \xrightarrow{\simeq} H_n(F/\gamma_2(R)) \otimes \Q.\]
\end{enumerate}
\end{fact}

\begin{prop}\label{prop:lim_hrp}
Let \(p\) be an odd prime, and \(1 \le r \le p\). Let \(G\) be a group without \((rp)!\)-torsion.
\begin{enumerate}
\item If \(r < p\) and \(i = 0, 1\), then
\[\lim^i \, H_{rp}(F/\gamma_2(R)) \otimes \Z_{(p)} = H_{r(p+2)-i}(G; \Z/p\Z).\]
\item If \(r = p\), then
\[\lim \, H_{p^2}(F/\gamma_2(R)) \otimes \Z_{(p)} = H_{p^2+2}(G; \Z/p\Z) \oplus H_{p^2+2p}(G; \Z/p\Z).\]
\end{enumerate}
\end{prop}
\begin{proof}
The result follows from Fact~\ref{fact:kuzmin_ex_rel_gamma_2} (2) and Proposition~\ref{prop:lim_ex_mod_p}.
\end{proof}

For a polynomial \(f(x) = \sum_{k \ge 0} m_k x^k\) with nonnegative integral coefficients and for any \(\ZG\)-module \(A\) we set
\[fH_n(G; A) = \bigoplus_{k \ge 0} H_{n+k}(G; A)^{\oplus m_k}.\]
For \(n \ge 2\) and a prime \(p\) we define \textit{Kuz'min polynomial \(f^{(p)}_n(x)\)} as
\[f^{(p)}_n(x) = \begin{cases}
0, & n \not\equiv 0, 1 \;\mathrm{mod}\; p, \\
x^2, & n = p, \\
x \cdot f^{(p)}_{n-1}(x), & n \equiv 1 \;\mathrm{mod}\; p, \\
x^2 \cdot f^{(p)}_{n-p}(x) + f^{(p)}_{n/p}(x), & n \equiv 0 \;\mathrm{mod}\; p, \; n > p.
\end{cases}\]
\begin{prop}\label{prop:hn_localized_away_2_kuzmin}
Let \(n \ge 2\), and \(G\) be a group without \(n!\)-torsion. Then there is an isomorphism
\[\lim \, H_n(F/\gamma_2(R)) \otimes \Z[1/2] = \bigoplus_p f^{(p)}_n H_n(G; \Z/p\Z),\]
where the sum runs over odd primes \(p\) dividing \(n\).
\end{prop}
\begin{proof}
Let \(T_n\) be a torsion part of \(H_n(F/\gamma_2(R))\).

In \cite[Corollary~1.1]{kovach_homology_1992}, it is proved that
\[T_n \otimes \Z[1/2] \simeq \bigoplus_p f^{(p)}_n H_n(G; \Z/p\Z).\]
By Fact~\ref{fact:kuzmin_ex_rel_gamma_2} (3), we have \(H_n(F/\gamma_2(R)) \otimes \Q \simeq (\Ex^n R_\ab)_G \otimes \Q\). This functor is monoadditive by Proposition~\ref{prop:lim_i_ex_rel_modulo_torsion}. The result follows from Lemma~\ref{lem:lim_commutes_with_torsion} (2).
\end{proof}

\begin{prop}\label{prop:torsionless_group_rationally}
Let \(n \ge 2\), and \(G\) be a \(n!\)-torsionless group. Then
\[\lim^i \, H_n(F/\gamma_2(R)) \otimes \Q = \begin{cases}
S^n_\Q(\g)_G, & i = n, \\
0, & i \ne n.
\end{cases}\]
\end{prop}
\begin{proof}
This follows from Fact~\ref{fact:kuzmin_ex_rel_gamma_2} (3) and Proposition~\ref{prop:lim_i_ex_rel_modulo_torsion}.
\end{proof}

\subsection{\(H_2(F/\gamma_*(R))\)}
Note that for every \(m \ge 1\) by Hopf's formula for the second homology one has
\begin{equation}\label{eq:h2_via_lie}
H_2(F/\gamma_m(R)) \cong \frac{[\overbrace{R, R, \ldots, R}^m]}{[\underbrace{R, R, \ldots, R}_m, F]} = \left(\frac{\gamma_m(R)}{\gamma_{m+1}(R)}\right)_F \cong (\Lie^m R_\ab)_G.
\end{equation}
The last isomorphism in the above series of identifications follows from the Magnus--Witt theorem; see Fact~\ref{fact:magnus_witt}.

The proof of the following proposition is tedious and can be found in Appendix~\ref{app:2}.
\begin{prop}\label{prop:lie_torsion}
Let \(G\) be a group and \(M\) be a \(\Z\)-free \(\ZG\)-module.
\begin{enumerate}
\item The group \(\lim \, (\Lie^n R_\ab \otimes M)_G\) is torsion for \(n \ge 3\).
\item If \(G\) is \(2\)-torsionless, then the groups \(\lim^i \, (\Lie^n R_\ab \otimes M)_G\) are torsion for \(n \ge 2\) and \(0 \le i < n\).
\end{enumerate}
\end{prop}

\begin{fact}[{\cite[Corollary~8.3]{hartley_homology_1991}, \cite[Theorem~6.3]{stohr_homology_1993}}]\label{fact:homology_lie}\leavevmode
\begin{enumerate}
\item Let \(p\) be a prime and \(G\) be a \(p\)-torsionless group. Then
\[H_k(G; \Lie^p R_\ab) \simeq H_{k+4}(G; \Z/p\Z), \quad k \ge 1,\]
and \(tH_0(G; \Lie^p R_\ab) \simeq H_4(G; \Z/p\Z)\).
\item Let \(G\) be a \(2\)-torsionless group. Then
\[H_k(G; \Lie^4 R_\ab) \simeq H_{k+6}(G; \Z/2\Z), \quad k \ge 1,\]
and \(tH_0(G; \Lie^4 R_\ab) \simeq H_6(G; \Z/2\Z)\).
\end{enumerate}
\end{fact}

\begin{cor}\label{cor:lim_h2_lie}
Let \(p\) be a prime and \(G\) be a \(p\)-torsionless group.
\begin{enumerate}
\item One has \[\lim\, H_2(F/\gamma_p(R)) = \lim \, (\Lie^p R_\ab)_G = H_4(G; \Z/p\Z).\]
\item If \(p = 2\), then
\[\lim\, H_2(F/\gamma_4(R)) = \lim \, (\Lie^4 R_\ab)_G = H_6(G; \Z/2\Z).\]
\end{enumerate}
\end{cor}
\begin{proof}
Follows from Proposition~\ref{prop:lie_torsion}, Fact~\ref{fact:homology_lie} and Lemma~\ref{lem:lim_commutes_with_torsion}~(1).
\end{proof}

\begin{rem}
Corollary~\ref{cor:lim_h2_lie}~(1) yields an alternative proof (avoiding the usage of metabelian Lie powers) to the isomorphism from \cite[p.~16]{ivanov_higher_2015}
\[\lim \, \frac{[R, R, R]}{[R, R, R, F]} = H_4(G; \Z/3)\]
for groups without \(3\)-torsion.
\end{rem}

We will return to these results in Section \ref{sec:fr_inf} when addressing the computation of limits of $\f\r_{\infty}$-codes.
\section{Higher limits of functorial ideals}
In this section, we establish key results that facilitate the computation of higher derived limits \(\lim^* \, \a\), where \(\a\) is a functorial ideal in a group ring \(\BaseRing[F]\) of a free extension \(F \epi G\). It turns out that Hopf ideals play a special role in this subject matter; see \S\ref{sec:hopf}.

Let \(\mathsf{FunIdeals}(\BaseRing[F])\) be a class of all two-sided ideals in \(\BaseRing[F]\) depending functorialy on a free extension \(F \epi G\). We observe that the most general functorial universums
\[\mathcal{U} \subset \mathsf{FunIdeals}(\BaseRing[F]),\]
amenable to productive study, are constructed via the following scheme. Take any finite set of functorial Hopf ideals \(\h_1, \ldots, \h_n\). We call these ideals \textit{letters}. Then, take its closure in \(\mathsf{FunIdeals}(\BaseRing[F])\) under finite sums, products and intersections.
Providing a general description of such languages presents an interesting theoretical challenge.

\subsection{Higher limits of free modules}
We start with a fundamental result that shows how to compute limits of functorial free modules.

Let \(\BaseRing\) be a commutative ring, and \(\BaseRing[-] \colon \Set \to \kMod\) be a free module functor. Let \(\C\) be a strongly connected category with binary coproducts.

\begin{thm}\label{thm:lim_free_module}
Let \(\U \colon \C \to \Set\) be a functorial set. Then
\[\lim^n \, \BaseRing[\U] = \begin{cases}
\BaseRing[\lim\, \U], & n = 0, \\
0, & n \ge 1.
\end{cases}\]
\end{thm}
\begin{proof}
Let \(\F = \BaseRing[\U] \colon \C \to \kMod\). Fix some \(c \in \C\). There is a standard cosimplicial object \(\BB(c)\) in \(\C\). Consider a cosimplicial abelian group \(\F\,\BB(c)\). We aim to compute \(\lim^n\,\F = \pi^n \, \F\,\BB(c)\). There are the alternate sum complex \(C = C\F\,\BB(c)\) and the Moore complex \(Q = Q\F\,\BB(c)\) of normalized cochains. There is a quasi-isomorphism \(C \epi Q\), and the \(n\)-th cohomology computes \(\lim^n\,\F\).

Let \(n = 0\). Then
\[\lim^0 \, \F = H^0(C) = \Zy^0(C) = \ker \{d^0 - d^1 \colon \BaseRing[\U(c)] \to \BaseRing[\U(c \sqcup c)]\}.\]
We have
\[\lim\,\U = \inv\,\U(c) = \eq\{\U(c) \xrightarrow{d^0, d^1} \U(c \sqcup c)\},\]
hence \(\BaseRing[\lim\,U]\) naturally embeds into \(\lim^0\, \F\). To show the reverse inclusion we proceed as follows. Let \(t \in \Zy^0(C)\) be a \(0\)-cocycle. We can write \(t = \sum_\alpha \lambda_\alpha x_\alpha\), where \(x_\alpha\) are some pair-wise distinct elements of \(\U(c)\), and \(\lambda_\alpha\) are non-zero elements of \(\BaseRing\). It is enough to show that for every \(\alpha\) one has \(d^0(x_\alpha) = d^1(x_\alpha)\). Note that if for some \(i, j \in \{0, 1\}\) and some indices \(\alpha, \beta\) one has an equality \(d^i(x_\alpha) = d^j(x_\beta)\), then
\[x_\alpha = s^0 d^i(x_\alpha) = s^0 d^j(x_\beta) = x_\beta.\]
The condition for \(t\) to be a cocycle deciphers as
\[\sum_\alpha \lambda_\alpha d^0(x_\alpha) = d^0(t) = d^1(t) = \sum_\alpha \lambda_\alpha d^1(x_\alpha).\]
Note that all elements \(\lambda_\alpha d^0(x_\alpha)\) of \(\F(c \sqcup c)\) are pair-wise distinct, since \(d^0\) is an embedding, and \(\F(c \sqcup c)\) is a free \(\BaseRing\)-module with the basis \(U(c \sqcup c)\). Hence, for every \(\alpha\) there is some index \(\beta\) such that \(d^0(x_\alpha) = d^1(x_\beta)\), which implies (by the observation above) that \(x_\alpha = x_\beta\), and shows the desired result. 

Now, let \(n \ge 1\). Then
\[\lim^n \, \F = H^n(C) = \Zy^n(C) / \Bd^n(C).\]
We aim to show that every \(n\)-cocycle \(t \in \Zy^n(C)\) is a coboundary. Since the natural projection \(C \epi Q\) is a quasi-isomorphism, it is enough to show that \(t\) maps to a coboundary of \(Q\). This is equivalent to show that \(t\) lies in the sum of the images of cofaces \(d^i \colon \F\,\BB(c)^{n-1} \to \F\,\BB(c)^n\).
As before, we write \(t = \sum_\alpha \lambda_\alpha x_\alpha\).
It is enough to show that each \(x_\alpha\) is in image of some \(d^i\).

The condition for \(t\) to be a cocycles deciphers as
\[0 = d(t) = \sum_{i = 0}^{n+1} (-1)^i d^i(t) = \sum_{i = 0}^{n+1} \sum_\alpha (-1)^i \lambda_\alpha d^i(x_\alpha).\]
Since \(\F(c^{\sqcup\,n+1})\) is a free \(k\){-}module with the basis \(\U(c^{\sqcup\,n+1})\), for every \(\alpha\) there exist \(i\) and \(\beta\) such that \((i, \beta) \ne (0, \alpha)\) and there is an identity
\begin{equation}\label{eq:lim_free_group_1}
d^0(x_\alpha) = d^i(x_\beta).
\end{equation}
There are three cases \(i = 0\), \(i = 1\) and \(i \ge 2\) which we consider separately.

First, if \(i = 0\), then \(x_\alpha = x_\beta\) (since \(d^0\) is monomorphism), and \(\alpha = \beta\) (since \(x\)'s are distinct), which contradicts \((i, \beta) \ne (0, \alpha)\).

Second, if \(i = 1\), then
\begin{equation}\label{eq:lim_free_group_2}
x_\alpha = s^0 d^0(x_\alpha) = s^0 d^1(x_\beta) = x_\beta,
\end{equation}
and
\[x_\alpha = s^1 d^1(x_\alpha) \overset{\eqref{eq:lim_free_group_2}}{=} s^1 d^1(x_\beta) \overset{\eqref{eq:lim_free_group_1}}{=} s^1 d^0(x_\alpha) = d^0 s^0(x_\alpha),\]
which shows that \(x_\alpha\) is in the image of \(d^0\).

Lastly, if \(i \ge 2\), then
\[x_\alpha = s^0 d^0(x_\alpha) \overset{\eqref{eq:lim_free_group_1}}{=} s^0 d^i(x_\beta) = d^{i-1} s^0(x_\beta),\]
which shows that \(x_\alpha\) is in the image of \(d^{i-1}\).
\end{proof}

\begin{cor}\label{cor:lim_free_aug}
Let \(\U \colon \C \to \Set\) be a functorial set. Then
\[\lim^n \, \Delta_\BaseRing(\U) = \begin{cases}
\Delta_\BaseRing(\lim\, \U), & n = 0, \\
0, & n \ge 1.
\end{cases}\]
\end{cor}
\begin{proof}
This follows from the long exact sequence of higher limits for the short exact sequence of representations
\[\Delta_\BaseRing(\U) \mono \BaseRing[\U] \epi \BaseRing,\]
and the Theorem~\ref{thm:lim_free_module}.
\end{proof}

\begin{cor}
Let \(\U\colon \C \to \Grp\) be a functorial group. Then
\[\lim^n \, \U_\ab \otimes \BaseRing = \lim^{n+1} \, \Delta_\BaseRing(\U)^2, \quad n \ge 1.\]
If \(\lim\, \U = 0\), then the above formula holds for \(n = 0\).
\end{cor}
\begin{proof}
This follows from the long exact sequence of higher limits for the short exact sequence of representations
\[\Delta_\BaseRing(\U)^2 \mono \Delta_\BaseRing(\U) \epi H_1(\U; \BaseRing) \simeq \U_\ab \otimes \BaseRing,\]
and the Corollary~\ref{cor:lim_free_aug}.
\end{proof}

\begin{lem}\label{lem:invariants_of_power}
Let \(\U\colon \C \to \Set\) be a functorial set.
Then for every \(n \ge 1\) one has \(\lim\, U^n = (\lim\, U)^n\) and \(\lim\, U^n/S_n = (\lim \, U)^n / S_n\).
\end{lem}
\begin{proof}
The first part is trivial, since limits commute with products.

We aim to prove the second part. We can identify the quotient \(U^n/S_n\) with the collection of multisets \(\{x_1, \ldots, x_n\}\) of elements \(x_i \in U\).
The inclusion \[(\lim \, U)^n / S_n \subset \lim\, U^n/S_n\]
is obvious. To prove the reverse inclusion, we proceed as follows. Let \(\{x_1, \ldots, x_n\}\) be some invariant of \(U^n/S_n\). Then there is an equality of multisets
\[\{d^0(x_1), \ldots, d^0(x_n)\} = \{d^1(x_1), \ldots, d^1(x_n)\}.\]
Then for every \(i\) there is some \(j\) such that \(d^0(x_i) = d^1(x_j)\). As before,
\[x_i = s^0 d^0(x_i) = s^0 d^1(x_j) = x_j,\]
and \(x_i\) is an invariant element of \(U\) as well.
\end{proof}
\begin{cor}\label{cor:power_of_free_limits}
Let \(\U\colon \C \to \Set\) be a functorial set, and \(\mathcal{F}\) be one of the following functors: \(\otimes^n\), \(S^n\). Then
\[\lim^i \, \mathcal{F}(\Z[U]) = \begin{cases}
\mathcal{F}(\Z[\lim\, U]), & i = 0, \\
0, & i > 0.
\end{cases}\]
\end{cor}
\begin{proof}
This follows from Theorem~\ref{thm:lim_free_module}, Lemma~\ref{lem:invariants_of_power}, and the canonical identifications \(T^n(\Z[U]) = \Z[U^n]\) and \(S^n(\Z[U]) = \Z[U^n / S_n]\).
\end{proof}

The following proposition says how to compute higher limits of some symmetric powers and exterior powers.
\begin{prop}
Let \(\U\colon \C \to \Set\) be a functorial set such that \(\lim \, U = *\) is a point, and \(\mathcal{F}\) be one of the following functors: \(\otimes^n\), \(S^n\), \(\Ex^n\), \(\Gamma^n\).
Then \[\lim^i \, \mathcal{F}(\Delta(U)) = 0\; i \ge 0.\]
\end{prop}
\begin{proof}
The result for \(\Delta(U)^{\otimes n}\) follows from the K\"unneth theorem.

We prove the result for \(\Ex^n(\Delta(U))\) by induction. The base case \(n = 1\) is obvious, since \(\Ex^1 = \id\). Now, let \(n \ge 2\).
Consider the Koszul sequence for the short exact sequence \(\Delta(U) \mono \Z[U] \epi \Z\):
\[\Ex^n(\Delta(U)) \mono \Ex^{n-1}(\Delta(U)) \otimes \Z[U] \to \cdots \to S^n(\Z[U]) \epi S^n(\Z).\]
All terms, except the first, have vanishing higher limits by the K\"unneth theorem, and the first page of the higher limits spectral sequence reads as
\[\begin{tikzcd}[cramped,sep=tiny]
	\vdots \\
	{\lim^2 \, \Ex^n(\Delta(U))} \\
	{\lim^1 \, \Ex^n(\Delta(U))} \\
	{\lim^0 \, \Ex^n(\Delta(U))} & 0 & \cdots & 0 & \Z & {\Z.}
	\arrow["\cong"{description}, draw=none, from=4-5, to=4-6]
\end{tikzcd}\]

The result for \(S^n(\Delta(U))\) and \(\Gamma^n(\Delta(U))\) can be proved similarly by considering appropriate Koszul sequences.
\end{proof}

\begin{ex}
Consider the Koszul sequence for the short exact sequences of free abelian groups \(\r \mono \f \epi \g\):
\begin{align*}
\Ex^n(\r) \mono \Ex^{n-1}(\r) \otimes \f \to \cdots \to S^n(\f) \epi S^n(\g), \\
\Gamma^n(\r) \mono \Gamma^{n-1}(\r) \otimes \f \to \cdots \to \Ex^n(\f) \epi \Ex^n(\g).
\end{align*}
Since all middle terms has vanishing limits by the previous proposition, we have
\[\lim^i \, \Ex^n(\r) = \begin{cases}
S^n(\g), & i = n, \\
0, & i \ne n,
\end{cases}\]
and
\[\lim^i \, \Gamma^n(\r) = \begin{cases}
\Ex^n(\g), & i = n, \\
0, & i \ne n,
\end{cases}\]
\end{ex}

\subsection{Product properties}

\begin{lem}\label{lem:adequate_prod}
Let \(\a\) and \(\b\) be ideals of \(\BaseRing[F]\) such that
\[\Tor^{\BaseRing[F]}_2(\BaseRing[F]/\a, \BaseRing[F]/\b) = 0.\]
Then
\begin{enumerate}
\item The multiplication map \(\a \otimes_{\BaseRing[F]} \b \to \a\b\) is an isomorphism.
\item For subideals \(\a' \subset \a\) and \(\b' \subset \b\) the multiplication map restricts to an isomorphism
\[(\a/\a') \otimes_{\BaseRing[F]} (\b/\b') \cong (\a\b) / (\a'\b + \a\b').\]
\end{enumerate}
\end{lem}
\begin{proof}
Straightforward generalization of \cite[Lemma~4.7, Lemma~4.8]{fr_codes_2017}.
\end{proof}

\begin{lem}\label{lem:check_adequate_prod}
Let \(\a\) and \(\b\) be two ideals of \(\BaseRing[F]\). Then one has
\[\Tor^{\BaseRing[F]}_2(\BaseRing[F]/\a, \BaseRing[F]/\b) = 0,\] if one of the following two conditions are satisfied:
\begin{enumerate}
\item \(\a\) (or \(\b\)) is a flat \(\BaseRing[F]\)-module;
\item one has \(\Tor^\BaseRing_1(\BaseRing[F]/\a, \BaseRing[F]/\b) = 0\) and \(\Tor^\BaseRing_2(\BaseRing[F]/\a, \BaseRing[F]/\a)_F = 0\).
\end{enumerate}
\end{lem}
\begin{proof}
The part (1) is obvious, since
\[\Tor^{\BaseRing[F]}_2(\BaseRing[F]/\a, \BaseRing[F]/\b) = \Tor^{\BaseRing[F]}_1(\a, \BaseRing[F]/\b).\]

Now, let's turn to the part (2).
Setting \(M = \BaseRing[F]/\a, N = \BaseRing[F]/\b\) and \(n = 2\) in Proposition~\ref{prop:tor_free_group} we conclude the result.
\end{proof}

The following proposition plays a crucial role in the computation of code
products.
\begin{prop}\label{prop:lim_of_product}
Let \(\a' \subset \a\) and \(\b' \subset \b\) be functorial ideals of \(\BaseRing[F]\) such that \(\a\) and \(\b\) are flat as \(\BaseRing[F]\)-modules.
\begin{enumerate}
\item If \(\lim^1\, \a'\b = \lim^1\, \a\b' = \lim\,\a\b = 0\),
then
\[\lim^1 \, \a'\b' = \lim \, \Tor^{\BaseRing[F]}_1(\a/\a', \b/\b').\]
\item If \(n \ge 1\) is such that
\[\lim^i\,\a'\b = \lim^i \,\a \b' = \lim^{i-1}\, \a\b = 0, \quad i \le n+1,\]
then there is an exact sequence
\begin{align*}
&\lim^1 \, \Tor^{\BaseRing[F]}_1(\a/\a', \b/\b') \mono \lim^2 \, \a' \b' \to \lim^1 \, \a'\b + \a \b' \\
\to & \lim^2 \, \Tor^{\BaseRing[F]}_1(\a/\a', \b/\b') \to \lim^3 \, \a'\b' \to \lim^2 \, \a'\b + \a\b' \to \cdots \\
\to &\lim^n \, \Tor^{\BaseRing[F]}_1(\a/\a', \b/\b') \to \lim^{n+1} \, \a'\b'.
\end{align*}
\end{enumerate}
\end{prop}
\begin{proof}
Since \(\b\) is flat, by Lemma~\ref{lem:check_adequate_prod}~(1), pairs \((\a', \b)\) and \((\a, \b)\) satisfy the condition of Lemma~\ref{lem:adequate_prod}.
There is a short exact sequence of representations \(\a' \mono \a \epi \a/\a'\). Tensoring it over \(\BaseRing[F]\) on the right with \(\b/\b'\), we conclude the exact sequence
\[\Tor^{\BaseRing[F]}_1(\a/\a', \b/\b') \mono \frac{\a'\b}{\a'\b'} \to \frac{\a\b}{\a\b'} \epi \frac{\a\b}{\a'\b + \a\b'}.\]

The result follows from Proposition~\ref{prop:spectral_application_2} 
\end{proof}

\begin{lem}\leavevmode
\begin{enumerate}
\item For any representation \(M\) of \(\BaseRing[F]\)-modules one has \(\lim \, H_1(F; M) = 0\).
\item For any representations \(M\) and \(N\) of \(\BaseRing[F]\)-modules one has
\[\lim \, \Tor^{\BaseRing[F]}_1(M, N) = \lim \, \Tor^\BaseRing_1(M, N)_F.\]
\end{enumerate}
\end{lem}
\begin{proof}
(1) This follows from the exact sequence
\[H_1(F; M) \mono \f \otimes_{\BaseRing[F]} M \to M \epi M_F,\]
split monoadditivity of \(\f \otimes_{\BaseRing[F]} M\) and the left exactness of \(\lim^0\).

(2) By Proposition~\ref{prop:tor_free_group}, there is a short exact sequence
\[\Tor^\BaseRing_1(M, N)_F \mono \Tor^{\BaseRing[F]}_1(M, N) \epi H_1(F; M \otimes_k N).\]
The result follows from (1) and the left exactness of \(\lim^0\).
\end{proof}

\begin{cor}
Let \(\a, \b \subset \f\) be functorial ideals. Then
\[\lim^1\, \a\b = \lim\, \Tor^\BaseRing_1(\f/\a, \f/\b)_F,\]
and there is an exact sequence
\begin{align*}
&\lim^1 \, \Tor^{\BaseRing[F]}_1(\f/\a, \f/\b) \mono \lim^2 \, \a\b \to \lim^1 \, \a\f + \f\b \\
\to &\lim^2 \, \Tor^{\BaseRing[F]}_1(\f/\a, \f/\b) \to \lim^3 \, \a\b \to \lim^2 \, \a\f + \f\b \to \cdots.
\end{align*}
\end{cor}

\begin{ex}
Let \(\a = \b = \f\f + \r\). Then \(\f/(\f\f + \r) = G_\ab\), and we have
\[\lim^1\, (\f\f+\r)^2 = \lim \, \Tor(G_\ab, G_\ab)_F = \Tor(G_\ab, G_\ab).\]
Also note that
\[\lim^1\, \f\f\f + \r\r = \inv(\f/(\f\f\f+\r\r)) = \frac{\f(\f\f + \r) \cap (\f\f+\r)\f}{\f\f\f + \r\r} = \Tor(G_\ab, G_\ab).\]
One can show that the inclusion \((\f\f + \r)^2 \mono \f\f\f+\r\r\) induces an isomorphism after applying the functor \(\lim^1\). This feature of the \(\f\r\)-language had been noticed by Fedor Pavutnitskiy; see \cite[p.~1587]{standard_complex_2020}.
\end{ex}

\subsection{Hopf ideals}
\label{sec:hopf}
A group homomorphism \(G \to K\) yields a homomorphism of algebras \(\BaseRing[G] \to \BaseRing[K]\).
This homomorphism preserves the natural structure of comultiplication, that is, this is actually a
homomorphism of Hopf algebras.

Let \(G\) be a group. If \(H\) is a normal subgroup of \(G\), we denote with \((H - 1)\BaseRing[G]\) the kernel of
the homomorphism \(\BaseRing[G] \to \BaseRing[G/H]\) induced by the projection \(G \epi G/H\).
If \(\a\) is an ideal of the
group algebra \(\BaseRing[G]\), we denote by \((\a + 1) \cap G\) the normal subgroup of \(G\) given by \(\{x \in G \mid x - 1 \in \a\}\).
Hence, we have two morphisms of posets
\[\begin{tikzcd}
	{\left\{\text{normal subgroups}\;H \subset G\right\}} && {\left\{\text{two-sided ideals}\; \a \subset \BaseRing[G]\right\}.}
	\arrow["{H\mapsto (H - 1)\BaseRing[G]}", shift left=2, from=1-1, to=1-3]
	\arrow["{(\a + 1) \cap G \mapsfrom \a}", shift left=2, from=1-3, to=1-1]
\end{tikzcd}\]
One may readily verify that the correspondence \(\a \mapsto (\a + 1) \cap G\) is a left inverse of \(H \mapsto (H - 1)\BaseRing[G]\), hence the set of normal subgroups embeds into the set of ideals of a group ring as a retract, with the image consisting of precisely Hopf ideals.

The following properties render Hopf ideals particularly significant for our investigation: in some sense they connect the \(\f\r\)-language --- studying ideals in a free group ring --- with the \(FR\)-language --- studying subgroups in a free group.
\begin{fact}[{\cite[Corollary~4.4]{fr_codes_2017}}]
Let \(H_1, \ldots, H_n\) be normal subgroups of a free group \(F\), and let \(\h_i = (H_i - 1)\BaseRing[F]\) be the corresponding Hopf ideals. Then the product \(\h_1 \h_2 \ldots \h_n\) is a free right (or left) \(\BaseRing[F]\)-module.
\end{fact}

\begin{fact}[{\cite[Lemma~4.5, Lemma~4.6]{fr_codes_2017}}]
Let \(H\) be a normal subgroup of \(F\), and \(\h = (H - 1)\BaseRing[F]\) be the corresponiding Hopf ideal.
\begin{enumerate}
\item The long sequence
\begin{equation}\label{eq:gruenberg_resolution}
\cdots \to \frac{\f\h^2}{\f\h^3} \to \frac{\h^2}{\h^3} \to \frac{\f\h}{\f\h^2} \to \frac{\h}{\h^2} \to \frac{\f}{\f\h} \to \frac{\BaseRing[F]}{\h} \to \BaseRing \to 0
\end{equation}
is a free resolution of the right \(\BaseRing[F/H]\)-module \(\BaseRing\).
\item For every \(n \ge 1\) there is an isomorphism
\[H_\ab^{\otimes n} \otimes \BaseRing \simeq \h^n / \h^n \f.\]
\item There is an isomorphism
\[\f/ \h \simeq \Delta_\BaseRing(F/H).\]
\end{enumerate}    
\end{fact}

\begin{prop}\label{prop:limits_of_letter}
Let \(\h = (H - 1)\BaseRing[F]\) be a functorial Hopf ideal in \(\BaseRing[F]\) given by some functorial subgroup \(H \subset F\). Then
\[\lim^n \, \h = \begin{cases}
\Delta_\BaseRing(\lim \, F/H), & n = 1, \\
0, & n \ne 1.
\end{cases}\]
\end{prop}
\begin{proof}
This follows from the long exact sequence of higher limits for the short exact sequence of representations
\[\h \mono \f \epi \frac{\f}{\h} = \Delta_\BaseRing(F/H),\]
and the Corollary~\ref{cor:lim_free_aug}.
\end{proof}

\begin{defn}\label{def:vacuous_letter}
A \textit{vacuous letter \(\h\)} is a functorial Hopf ideal in \(\BaseRing[F]\) given by some functorial subgroup \(H \subset F\) such that \(\lim \, F/H = 0\).
\end{defn}

\begin{prop}\label{prop:limits_vacuous}
Let \(\BaseRing\) be a heriditary commutative ring, \(\h\) be a vacuous letter. If \(\a, \b \subset \BaseRing[F]\) are finitary functorial ideals such that one of them is flat as \(\BaseRing[F]\)-module, then \(\lim^n \, \a\h\b = 0\) for every \(n \ge 0\).
\end{prop}
\begin{proof}
Since \(\h\) is free \(\BaseRing[F]\)-module, one has \(\a\h\b \cong \a \otimes_{\BaseRing[F]} \h \otimes_{\BaseRing[F]} \b\).
The result follows from Corollary~\ref{cor:kunneth_appl} and Proposition~\ref{prop:limits_of_letter}.
\end{proof}

\section{\(\f\r_\infty\)-language}
\label{sec:fr_inf}
To the contrast of the case of \(\f\r\)-language it is now essential to use other rings than integers.
The reasons may be clearly seen from the sections~\ref{sec:shifting} and \ref{sec:quotients} which indicate that we need to localize \(\Z\) to calculate certain higher limits.

The free group algebra functor \(\BaseRing[-] \colon \Pres \to \Ring\) has a chain of functorial ideals
\[\BaseRing[-] \supset \f \supset \r_1 \supset \r_2 \supset \r_3 \supset \ldots\]
given by
\[\r_m(R \mono F \epi G) = (\gamma_m(R) - 1)\BaseRing[F].\]
Note that \(\r_1 = \r\) in the old notation.

\begin{lem}
For every \(m \ge 2\) one has \(\lim \, F/\gamma_m(R) = 0\).
\end{lem}
\begin{proof}
We prove the statement by induction over \(m\).

For \(m = 2\) consider a short exact sequence
\[R_\ab \mono F/\gamma_2(R) \epi F/R = G.\]
Since \(\lim^0\) is left exact, \(\lim \, R_\ab = 0\) by \cite[Lemma~4.14]{fr_codes_2017}, there is a monomorphism \(\lim \, F/\gamma_2(R) \mono F/R\). Since \(\Pres(G)\) is strongly connected, the limit could be computed as the invariants:
\[\lim \, F/\gamma_2(R) = \inv \, F/\gamma_2(R) = H/\gamma_2(R),\]
where \(H\) is the maximal functorial subgroup of \(F\), such that the quotient \(H/\gamma_2(R)\) is constant. By functorial properties of the limit, \(H\) is normal in \(F\). The monomorphism \(H/\gamma_2(R) \mono F/R\) implies that \(R \cap H = \gamma_2(R)\). But then \(H/(R \cap H) \subset G\) acts trivially on the relation module \(R_\ab\). The action of \(G\) on \(R_\ab\) is faithful by \cite[Theorem~1]{auslander_commutator_1955}, which implies that \(H = R \cap H = \gamma_2(R)\).

For \(m \ge 3\) consider a short exact sequence
\[\Lie^{m-1} R_\ab \mono F/\gamma_m(R) \epi F/\gamma_{m-1}(R).\]
One has \(\lim \, \Lie^{m-1} R_\ab \mono \lim \, R_\ab^{\otimes m-1} = 0\) by \cite[Lemma~4.14]{fr_codes_2017}. The result follows from the left exactness of \(\lim^0\) and the induction hypothesis.
\end{proof}
Thus, \(\r_m\) are vacuous letters for \(m \ge 2\) in the sense of Definition~\ref{def:vacuous_letter}. Hence,
\[\lim^n \, \r_m = \begin{cases}
\g, & (n,m) = (1,1), \\
0, & \text{otherwise}.
\end{cases}\]
Moreover, for \(m \ge 2\) and every finitary functorial ideals \(\a, \b \subset k[F]\) such that one of them is flat, the limits \(\lim^*\, \a\r_m\b\) are zero by Proposition~\ref{prop:limits_vacuous}.

Note that for every \(m \ge 2\) from~\eqref{eq:gruenberg_resolution} one has
\begin{align}
H_{2k}(F/\gamma_m(R)) &= \frac{\r_m^k \cap \f\r_m^{k-1}\f}{\r_m^k \f + \f\r_m^k}, \; k \ge 1, \nonumber \\
\label{eq:homology_via_frinf}
H_{2k+1}(F/\gamma_m(R)) &= \frac{\f\r_m^k \cap \r_m^k \f}{\r_m^{k+1} + \f\r_m^k\f}, \; k \ge 0.
\end{align}
Clearly, the numerators have trivial \(\lim^1\). Therefore,
\begin{align*}
&\lim^1\, \r_m^k \f + \f\r_m^k = \lim\,H_{2k}(F/\gamma_m(R)), \; k \ge 1, \\
&\lim^1\, \r_m^{k+1} + \f\r_m^k\f = \lim\, H_{2k+1}(F/\gamma_m(R)), \; k \ge 0.
\end{align*}

This allows us to write down the computations from Section~\ref{sec:quotients} in the following compact form.
\begin{table}[H]
\resizebox{\textwidth}{!}{%
\begin{tabular}{|c|c|c|c|}
\hline
\(G\)                                                                                                      & \(\a \subset \BaseRing[F]\)                         & \(\BaseRing\)        & \(\lim^1 \, \a\)                                                  \\ \hline
\(p\)-torsionless for prime \(p\)                                                                          & \(\r_p\f + \f\r_p\)                     & \(\Z\)       & \(H_4(G; \Z/p)\)                                                             \\ \hline
\(2\)-torsionless                                                                                          & \(\r_4\f + \f\r_4\)                     & \(\Z\)       & \(H_6(G; \Z/2)\)                                                             \\ \hline
\((2n)!\)-torsionless for \(n \ge 1\)                                                                      & \(\s^n \f + \f \s^n\)                   & \(\Z[1/2]\)  & \(\bigoplus\limits_{\text{odd}\, p\,\mid\,n} f^{(p)}_{2n} H_{2n}(G; \Z/p)\)     \\ \hline
\((2n+1)!\)-torsionless for \(n \ge 1\)                                                                    & \(\s^{n+1} + \f\s^n\f\)                 & \(\Z[1/2]\)  & \(\bigoplus\limits_{\text{odd}\, p \,\mid\,n} f^{(p)}_{2n+1} H_{2n+1}(G; \Z/p)\) \\ \hline
\begin{tabular}[c]{@{}c@{}}\((rp)!\)-torsionless\\ for odd prime \(p\) and even \(0 < r < p\)\end{tabular} & \(\s^{rp/2} \f + \f\s^{rp/2}\)          & \(\Z_{(p)}\) & \(H_{r(p+2)}(G; \Z/p)\)                                                      \\ \hline
\begin{tabular}[c]{@{}c@{}}\((rp)!\)-torsionless\\ for odd prime \(p\) and odd \(0 < r < p\)\end{tabular}  & \(\s^{(rp+1)/2}\f + \f\s^{(rp-1)/2}\f\) & \(\Z_{(p)}\) & \(H_{r(p+2)}(G; \Z/p)\)                                                      \\ \hline
\begin{tabular}[c]{@{}c@{}}\(p!\)-torsionless\\ for odd prime \(p\)\end{tabular}                           & \(\s^{(p^2+1)/2} + \f\s^{(p^2-1)/2}\f\) & \(\Z_{(p)}\) & \(H_{p^2+2}(G;\Z/p) \oplus H_{p^2+2p}(G;\Z/p)\)                              \\ \hline
\end{tabular}%
}
\end{table}
\begin{itemize}
\item Rows (1) and (2) follow from Corollary~\ref{cor:lim_h2_lie}.
\item Rows (3) and (4) follow from Proposition~\ref{prop:hn_localized_away_2_kuzmin}.
\item Rows (5), (6) and (7) follow from rows (3) and (4). Alternatively, see Proposition~\ref{prop:lim_hrp}.
\end{itemize}

Note that if \(\a\) and \(\b\) have vanishing limits, then \(\lim^i\, \a \cap \b = \lim^{i-1}\,\a + \b\) for every \(i \ge 1\), because of the short exact sequence \(\a \cap \b \mono \a \oplus \b \epi \a + \b\).
Hence, one has
\begin{align}
&\lim^i \, \r_m^k \cap \f\r_m^{k-1}\f = \lim^{i-1} \, \r_m^k + \f\r_m^{k-1}\f, \quad i \ge 1,
\nonumber \\
\label{eq:shift_gruenberg_numerators}
&\lim^i \, \f\r_m^k \cap \r_m^k \f = \lim^{i-1} \, \f\r_m^k + \r_m^k\f, \quad i \ge 1.
\end{align}
The long exact sequence of limits for \eqref{eq:homology_via_frinf} yield the following two sequences:
\begin{align*}
&\lim^1\, H_{2k}(F/\gamma_m(R)) \mono \lim^2 \, \r_m^k \f + \f \r_m^k \to \lim^1 \, \r_m^k + \f\r_m^{k-1}\f \\
\to &\lim^2\, H_{2k}(F/\gamma_m(R)) \to \lim^3 \, \r_m^k \f + \f \r_m^k \to \lim^2 \, \r_m^k + \f\r_m^{k-1}\f \to \cdots
\end{align*}
and 
\begin{align*}
&\lim^1\, H_{2k+1}(F/\gamma_m(R)) \mono \lim^2 \, \r_m^{k+1}+ \f \r_m^k\f \to \lim^1 \, \f \r_m^k + \r_m^k\f \\
\to &\lim^2\, H_{2k+1}(F/\gamma_m(R)) \to \lim^3 \, \r_m^{k+1}+ \f \r_m^k\f \to \lim^2 \, \f \r_m^k + \r_m^k\f \to \cdots.
\end{align*}
Even for \(m = 2\) the full computation of \(\lim^2\) for Gruenberg ideals seems a hard problem, but some formulas with localized coefficients can be obtained.
\begin{prop}
Let \(p\) be an odd prime, and \(0 < r < p\) be an integer. Let \(G\) be a group without \((rp)!\)-torsion. Then there are the following formulas relating functorial ideals \(\s = (\gamma_2(R)-1)\Z_{(p)}[F]\) and \(\f = (F-1)\Z_{(p)}[F]\).
\begin{enumerate}
\item If \(r\) is even, then
\[\lim^2_{\Z_{(p)}} \, \s^{rp/2}\f + \f \s^{rp/2} = H_{r(p+2)-1}(G; \Z/p).\]
\item If \(r\) is odd, then
\[\lim^2_{\Z_{(p)}} \,\s^{(rp+1)/2}\f +\f\s^{(rp-1)/2}\f = H_{r(p+2)-1}(G; \Z/p).\]
\end{enumerate}
\end{prop}
\begin{proof}
This follows from the above exact sequences localized at \(p\) and Proposition~\ref{prop:lim_hrp}~(1).
\end{proof}

To gain a deeper understanding of codes in the \(\f\r_\infty\)-language, it is essential to conduct computations over \(\Q\). The following proposition provides the calculation for all higher limits of Gruenberg ideals involving \(\s = \r_2\).
\begin{prop}
Let \(G\) be a torsionless group. One has
\[\lim^i\, \s^k \f + \f \s^k \otimes \Q = \begin{cases}
\bigoplus\limits_{2 \le j \le 2k} S^j_\Q(\g)_G, & i = 2k + 1, \\
0, & i \ne 2k+1,
\end{cases}\]
and
\[\lim^i\, \s^k + \f \s^{k-1} \f \otimes \Q = \begin{cases}
\bigoplus\limits_{2 \le j < 2k} S^j_\Q(\g)_G, & i = 2k, \\
0, & i \ne 2k,
\end{cases}\]
\end{prop}
\begin{proof}
The formulas \eqref{eq:homology_via_frinf} and \eqref{eq:shift_gruenberg_numerators} specialize in the case \(m = 2\) to the following:
\begin{align*}
H_{2k}(F/\gamma_2(R)) &= \frac{\s^k \cap \f\s^{k-1}\f}{\s^k \f + \f\s^k}, \; k \ge 1, \\
H_{2k+1}(F/\gamma_2(R)) &= \frac{\f\s^k \cap \s^k \f}{\s^{k+1} + \f\s^k\f}, \; k \ge 0,
\end{align*}
and
\begin{align*}
&\lim^i \, \s^k \cap \f\s^{k-1}\f = \lim^{i-1} \, \s^k + \f\s^{k-1}\f, \quad i \ge 1, \\
&\lim^i \, \f\s^k \cap \s^k \f = \lim^{i-1} \, \f\s^k + \s^k\f, \quad i \ge 1.
\end{align*}
Also, note that by Proposition~\ref{prop:torsionless_group_rationally} for \(k \ge 2\) one has
\[\lim^i \, H_k(F/\gamma_2(R)) \otimes \Q = \begin{cases}
S^k_\Q(\g)_G, & i = k, \\
0, & i \ne k.
\end{cases}\]
Therefore, there are exact sequences
\begin{align*}
&\underbrace{\lim^{2k} \, \s^k \f + \f \s^k \otimes \Q}_{A_k'} \mono \underbrace{\lim^{2k-1} \, \s^k + \f\s^{k-1}\f \otimes \Q}_{B_k'} \\
\to &S^{2k}_\Q(\g)_G \to \underbrace{\lim^{2k+1} \, \s^k \f + \f\s^k \otimes \Q}_{A_k} \epi \underbrace{\lim^{2k} \s^k + \f\s^{2k-1}\f \otimes \Q}_{B_k}, \quad k \ge 1, \\
& \underbrace{\lim^{2k+1} \, \s^{k+1} + \f\s^k \f \otimes \Q}_{B_{k+1}'} \mono \underbrace{\lim^{2k}\, \f\s^k + \s^k\f \otimes \Q}_{A_k'} \\
\to & S^{2k+1}_\Q(\g)_G \to \underbrace{\lim^{2k+2}\, \s^{k+1} + \f\s^k \f \otimes \Q}_{B_{k+1}} \epi \underbrace{\lim^{2k+1}\, \f\s^k + \s^k\f \otimes \Q}_{A_k}, \quad k \ge 0.
\end{align*}
Since there is a chain of inclusions
\[\cdots \mono B_3' \mono A_2' \mono B_2' \mono A_1' \mono B_1',\]
and
\[B_1' = \lim^1\, \s + \f\f \otimes \Q = \lim \, H_1(F/\gamma_2(R)) \otimes \Q = \lim \, F_\ab \otimes \Q = 0,\]
we have \(A_k' = B_k' = 0\) for all \(k\).
Therefore, there are short exact sequences
\begin{align*}
&S^{2k}_\Q(\g)_G \mono \underbrace{\lim^{2k+1} \, \s^k \f + \f\s^k \otimes \Q}_{A_k} \epi \underbrace{\lim^{2k} \s^k + \f\s^{2k-1}\f \otimes \Q}_{B_k}, \quad k \ge 1, \\
&S^{2k+1}_\Q(\g)_G \mono \underbrace{\lim^{2k+2}\, \s^{k+1} + \f\s^k \f \otimes \Q}_{B_{k+1}} \epi \underbrace{\lim^{2k+1} \, \s^k \f + \f\s^k \otimes \Q}_{A_k}, \quad k \ge 0.
\end{align*}
Since over \(\Q\) every exact sequence splits, we have
\begin{align*}
A_k &= S^{2k}_\Q(\g)_G \oplus S^{2k-1}_\Q(\g)_G \oplus \cdots \oplus S^2_\Q(\g)_G, \quad k \ge 0, \\
B_k &= S^{2k-1}_\Q(\g)_G \oplus S^{2k-2}_\Q(\g)_G \oplus \cdots \oplus S^2_\Q(\g)_G, \quad k \ge 1,
\end{align*}
and the result follows.
\end{proof}
\appendix
\renewcommand{\thesection}{\Alph{section}}
\section{Proof of technical results}

\subsection{Proof of Proposition~\ref{prop:s2_p_monoadditive}}
\label{app:1}

Let \(n = 0\). Then there are two exact sequences
\begin{align*}
& S^2(P) \mono \Gamma^2(P) \epi P \otimes \Z/2, \\
& \Gamma^2(P) \mono P \otimes P \epi S^2(P).
\end{align*}
Since \(P\) and \(S^2(P)\) are free \(\ZG\)-modules, the modules \(P \otimes M\) and \(S^2(P) \otimes M\) are free as well, and these sequences induce short exact sequences
\begin{align*}
& (S^2(P) \otimes M)_G \mono (\Gamma^2(P) \otimes M)_G \epi (P \otimes M \otimes \Z/2)_G, \\
& (\Gamma^2(P) \otimes M)_G \mono (P \otimes P \otimes M)_G \epi (S^2(P) \otimes M)_G.
\end{align*}
Hence,
\[\lim \, (S^2(P) \otimes M)_G \mono \lim \, (\Gamma^2(P) \otimes M)_G \mono \lim \, (P \otimes P \otimes M)_G = 0,\]
since the last functor is split monoadditive by Lemma~\ref{lem:h0_m_o_p_monoadditive} (1).

If \(n = 1\), then the result holds by Lemma~\ref{lem:h0_m_o_p_monoadditive} (3).

Now, assume that \(n \ge 2\).
Tensor \eqref{eq:ex2p_flat} over \(\Z\) with \(M\), and then tensor the result over \(\ZG\) with \(\Ex^n R_\ab\), to obtain the following exact sequence:
\[(\Ex^n R_\ab \otimes \Ex^2(P) \otimes M)_G \mono (\Ex^n R_\ab \otimes P \otimes P \otimes M)_G \epi (\Ex^n R_\ab \otimes S^2(P) \otimes M)_G.\]

The middle has trivial limits by Lemma~\ref{lem:h0_m_o_p_monoadditive} (1), and
\[\lim \, (\Ex^n\, R_\ab \otimes S^2(P) \otimes M)_G = \lim^1 \, (\Ex^n\, R_\ab \otimes \Ex^2(P) \otimes M)_G.\]

From \eqref{eq:ex2p_flat} we can see that \(\Ex^2(P)\) is flat \(\ZG\)-module. Hence, we can tensor the Kozsul sequence for the free abelian group \(S^n R_\ab\) with \(\Ex^2(P) \otimes M\) over \(\ZG\) to obtain the following exact sequence:
\begin{align*}
(\Ex^n R_\ab \otimes \Ex^2(P)& \otimes M)_G \mono (\Ex^{n-1} R_\ab \otimes R_\ab \otimes \Ex^2(P) \otimes M)_G \to \cdots \\
&\epi (S^n R_\ab \otimes \Ex^2(P) \otimes M)_G.
\end{align*}
Hence, it is enough to show that
\[\lim^1 \, (\Ex^{n-1} R_\ab \otimes R_\ab \otimes \Ex^2(P) \otimes M)_G = 0\]
and
\[\lim\, (\Ex^{n-2} R_\ab \otimes S^2(R_\ab) \otimes \Ex^2(P) \otimes M)_G = 0.\]

The second identitity follows from the Lemma~\ref{lem:h0_m_o_p_monoadditive} (2). 

To show the first one, we tensor \eqref{eq:ex2p_flat} with \(M\) over \(\Z\), and with \(\Ex^{n-1} R_\ab \otimes R_\ab\) over \(\ZG\), and obtain the short exact sequence
\begin{align*}
(\Ex^{n-1} R_\ab \otimes R_\ab& \otimes \Ex^2(P) \otimes M)_G \mono (\Ex^{n-1} R_\ab \otimes R_\ab \otimes P \otimes P \otimes M)_G \\
&\epi (\Ex^{n-1} R_\ab \otimes R_\ab \otimes S^2(P) \otimes M)_G.
\end{align*}
Since the middle term is split monoadditive by Lemma~\ref{lem:h0_m_o_p_monoadditive} (1), it is enough to show that
\[\lim\, (\Ex^{n-1} R_\ab \otimes R_\ab \otimes S^2(P) \otimes M)_G = 0.\]
But this holds by Lemma~\ref{lem:h0_m_o_p_monoadditive} (3).
\qed

\subsection{Proof of Proposition~\ref{prop:lie_torsion}}
\label{app:2}

The following perspective to Lie powers was shown to us by Artem Semidetnov.
For \(n \ge 1\) and a free abelian group \(A\) the tensor power
\[T^n A = \underbrace{A \otimes \ldots \otimes A}_n\]
is a module over \(\Z[S_n]\) in a natural way. If \(\a\) is some left ideal of \(\Z[S_n]\), then \(T^n A \cdot \a\) is a left \(\Z[S_n]\)-submodule of \(T^n A\).

\begin{ex}
Let
\[x = (1 - (2,1)) \cdot (1 - (3,2,1)) \cdot \ldots \cdot (1 - (n, n-1, \ldots, 1)),\]
and \(\a = \Z[S_n] x\). Then \(T^n A \cdot \a = T^n A \cdot x\) is isomorphic to the \(n\)-th Lie power \(\Lie^n A\).
\end{ex}

If \(\a\) and \(\b\) are two ideals of \(\Z[S_n]\), then every homomorphism of left \(\Z[S_n]\)-modules \(\a \to \b\) gives a natural transformation \(T^n A \cdot \a \to T^n A \cdot \b\) in the functor category \(\Fun(\Ab_\mathrm{free}, \Mod(\Z[S_n]))\).

Let \(x, y\) be any two elements of the group ring \(\Z[S_n]\). Then there are two homomorphisms of left \(\Z[S_n]\)-modules: the injection
\(\Z[S_n] yx \mono \Z[S_n] x\) and the surjection \(\Z[S_n] x \epi \Z[S_n] xy\).

\begin{ex}
Consider two elements \(x, y \in \Z[S_n]\) given by
\begin{align*}
x &= (1 - (2,1)) \cdot \ldots \cdot (1 - (n-1,\ldots, 1)), \\
y &= (1 - (n,\ldots,1)).
\end{align*}
Then the homomorphism \(\Z[S_n] x \epi \Z[S_n] xy\) corresponds to the natural epimorphism
\(\Lie^{n-1} \otimes \mathrm{Id} \epi \Lie^n\) given by the formula
\[[x_1, \ldots, x_{n-1}] \otimes x \mapsto [x_1, \ldots, x_{n-1}, x].\]
This map appears as the component of the last map in the Chevalley-Eilenberg complex 
\[\Ex^n \mono \ldots \to \bigoplus_{\substack{k_1 + \ldots + k_n = n-1 \\ k_1 + 2\cdot k_2 + \ldots + n \cdot k_n = n}} \Ex^{k_1} \Lie^1 \otimes \cdots \otimes \Ex^{k_n} \Lie^n \epi \Lie^n.\]
\end{ex}

\begin{lem}\label{lem:semidetnov_splitting}
For every elements \(x, y \in \Z[S_n]\), after tensoring with \(\Q\), the natural epimorphism \((T^n A \cdot x) \otimes \Q \epi (T^n A \cdot xy) \otimes \Q\) splits, and the natural monomorphism \((T^n A \cdot yx) \otimes \Q \mono (T^n A \cdot x) \otimes \Q\) splits.
\end{lem}
\begin{proof}
We prove only the `epi' part. The idea is the same as in the proof of Maschke's theorem.

Note that for every \(t \in \Z[S_n]\) there is a natural isomorphism \[(T^n A \cdot t) \otimes \Q = (T^n A \otimes \Q) \cdot t,\] where `\(t\)' on the right is considered as an element of rational group ring \(\Q[S_n]\).
Hence, it is enough to construct the splitting of \(\Q[S_n]\)-modules for \(f \colon \Q[S_n]x \epi \Q[S_n]xy\). Since this is epimorphism of \(\Q\)-modules, there exists a injective homomorphism of vector spaces \(g \colon \Q[S_n]xy \mono \Q[S_n]x\) such that \(f \circ g = \id\). Consider a homogenization
\[\widehat g \colon \Q[S_n]xy \to \Q[S_n]x, \quad \widehat g(u) = \tfrac{1}{n!} \sum_{\sigma \in S_n} \sigma \cdot g(\sigma^{-1} u).\]
It is easy to see that \(f \circ \widehat g = \id\) and \(\widehat g\) is a homomorphism of \(\Q[S_n]\)-modules.
\end{proof}

Now we are ready to prove Proposition~\ref{prop:lie_torsion}.
By Lemma~\ref{lem:lim_commutes_with_torsion} (2), it is enough to show that abelian groups \(\lim^i \, (\Lie^n R_\ab \otimes M)_G \otimes \Q\) are zero in the specified range of \(i\).

(1) Consider a composition
\[\Lie^n R_\ab \mono T^n R_\ab \simeq T^{n-1} R_\ab \otimes R_\ab \mono T^{n-1} R_\ab \otimes P.\]
Applying \(H_0(G; - \otimes M)\) to it yields a map \((\Lie^n R_\ab \otimes M)_G \to (T^{n-1} R_\ab \otimes P \otimes M)_G\). By \cite[Lemma~2.1 (i)]{stohr_homology_1993}, its kernel is torsion, and there is a monomorphism
\[(\Lie^n R_\ab \otimes M)_G \otimes \Q \mono (R_\ab \mono T^{n-1} R_\ab \otimes P \otimes M)_G \otimes \Q.\]
The result follows since \(\lim^0\) is left exact, and the right hand side is split monoadditive by Lemma~\ref{lem:h0_m_o_p_monoadditive} (1).

(2)
We prove this by induction.
This holds for \(n = 2\) by Proposition~\ref{prop:lim_i_ex_rel_modulo_torsion}, since \(\Lie^2 = \Ex^2\).

By Lemma~\ref{lem:semidetnov_splitting}, there is a split monomorphism \(\Lie^n \otimes \Q \mono \Lie^{n-1} \otimes \mathrm{Id} \otimes \Q\). It induces monomorphisms
\[\lim^i\, (\Lie^n R_\ab \otimes M)_G \otimes \Q \mono \lim^i\, (\Lie^{n-1} R_\ab \otimes R_\ab \otimes M)_G \otimes \Q, \quad i \ge 0.\]
Tensoring a short exact sequence of representations \eqref{eq:relation_module} with \(\Lie^{n-1} R_\ab \otimes M \otimes \Q\) over \(\Z[G]\) yields a short exact sequence
\begin{align*}
(\Lie^{n-1} R_\ab \otimes R_\ab &\otimes M)_G \otimes \Q \mono (\Lie^{n-1} R_\ab \otimes P \otimes M)_G \otimes \Q \\
&\epi (\Lie^{n-1} R_\ab \otimes \g \otimes M)_G \otimes \Q,
\end{align*}
since \(H_1(G; \Lie^{n-1} R_\ab \otimes \g \otimes M \otimes \Q)\) is zero by \cite[Theorem~4.1]{zerck_homology_1989}\footnote{The cited theorem states that \(H_1(G; \Lie^{n-1}R_\ab)\) is torsion, but this easily generalizes to \(H_1(G; \Lie^{n-1}R_\ab \otimes D)\) for any \(\Z\)-free \(\ZG\)-module \(D\), such as \(D = \g \otimes M\).}.
The middle term of this sequence is split monoadditive by Lemma~\ref{lem:h0_m_o_p_monoadditive} (1), and the result follows from the inductive hypothesis. 
\qed

\end{document}